\documentclass[10pt,a4paper,british,a4wide]{scrartcl}
\usepackage[T1]{fontenc}
\usepackage[latin1]{inputenc}
\usepackage{babel}
\usepackage{verbatim}
\usepackage{prettyref}
\usepackage{mathtools}
\usepackage{amsmath}
\usepackage{amsthm}
\usepackage{amssymb}
\usepackage{esint}
\usepackage[unicode=true,
 bookmarks=true,bookmarksnumbered=false,bookmarksopen=false,
 breaklinks=false,pdfborder={0 0 1},backref=false,colorlinks=false]
 {hyperref}

\makeatletter

\pdfpageheight\paperheight
\pdfpagewidth\paperwidth

\theoremstyle{definition}
\newtheorem*{defn*}{\protect\definitionname}
\theoremstyle{plain}
\newtheorem{thm}{\protect\theoremname}[section]
\theoremstyle{plain}
\newtheorem{prop}[thm]{\protect\propositionname}
\theoremstyle{remark}
\newtheorem{rem}[thm]{\protect\remarkname}
\theoremstyle{plain}
\newtheorem{cor}[thm]{\protect\corollaryname}
\theoremstyle{plain}
\newtheorem{lem}[thm]{\protect\lemmaname}
\theoremstyle{definition}
\newtheorem{example}[thm]{\protect\examplename}

\@ifundefined{date}{}{\date{}}
\usepackage[T1]{fontenc}    
\usepackage{a4wide}        
\usepackage{amsmath}
\usepackage{enumerate}
\usepackage{mathtools}
\usepackage{euscript}
\allowdisplaybreaks[1] 
\usepackage{amsfonts}

\newcommand*{\dom}{\operatorname{dom}}  
\newcommand*{\ran}{\operatorname{ran}}            
\newcommand{\R}{\mathbb{R}}

\newcommand*{\dive}{\operatorname{div} \,}

\renewcommand{\d}{\,\mathrm{d}}

\newcommand*{\grad}{\operatorname{grad} \,}

\newcommand*{\supp}{\operatorname{supp}}

\newcommand*{\e}{\mathrm{e}}

\DeclareMathAccent{\Circ}{\mathalpha}{operators}{"17}
\newcommand{\interior}[1]{\Circ{#1}}
\DeclareMathAccent{\fiss}{\mathalpha}{operators}{"15}

\newcommand{\sym}{\operatorname{sym}}
\renewcommand{\skew}{\operatorname{skew}}

\newcommand{\N}{\mathbb{N}}
\newcommand{\spt}{\operatorname{spt}}
\renewcommand*{\epsilon}{\varepsilon}
\renewcommand*{\theta}{\vartheta}
\renewcommand{\tilde}{\widetilde}

\theoremstyle{definition}
\newtheorem{hyp}{Hypotheses}

\AtBeginDocument{

}

\makeatother

\providecommand{\corollaryname}{Corollary}
\providecommand{\definitionname}{Definition}
\providecommand{\examplename}{Example}
\providecommand{\lemmaname}{Lemma}
\providecommand{\propositionname}{Proposition}
\providecommand{\remarkname}{Remark}
\providecommand{\theoremname}{Theorem}

\begin{document}
\title{{\Huge{}Dynamic First Order Wave Systems with Drift Term on Riemannian
Manifolds.}}
\author{Rainer Picard\thanks{Institut für Analysis, TU Dresden, Germany}
and Sascha Trostorff\thanks{Mathematisches Seminar, CAU Kiel, Germany}}
\maketitle
\begin{abstract}
An abstract first order differential equation of hyperbolic type with
drift term on a Riemannian manifold is considered. For proving its
well-posedness, transmutator and commutator relations are needed,
which are studied in a general functional analytic setting. 
\end{abstract}

\section{Introduction}

In isotropic and homogeneous media linear acoustic waves are governed
by a system combining Euler's force equation (momentum balance)
\begin{equation}
D_{t}v+\grad p=0\label{eq:Euler}
\end{equation}
with the continuity equation (mass balance)
\begin{equation}
D_{t}p+\dive v=f,\label{eq:Cont}
\end{equation}
where we have for simplicity reduced all parameters by rescaling (in
effect having in particular $1$ as the speed of sound). Here $p$
denotes the pressure and $v$ the velocity field associates with the
acoustic wave, $f$ is a given source term. In a rest frame situation
$D_{t}=\partial_{0}$, the partial derivative with respect to time.
In moving media, however, we have -- by a suitable rotation of coordinates
in $\mathbb{R}^{3}$ -- $D_{t}=\partial_{0}+v_{0}\partial_{3}$,
with $v_{0}\in\mathbb{R}$ denoting the velocity of the drift with
direction $e_{3}=\left(0,0,1\right)$ of the underlying media, the
so-called convective, substantial or material derivative. Assuming
with much loss of generality an irrotational velocity field we can
reduce this to a bi-isotropic, homogeneous, but otherwise standard
acoustic wave system\footnote{This is mimicking the approach of constructing the Maxwell-Hertz-Cohn
system of ``pre-relativity'' electrodynamics in moving media.} by introducing
\[
\left(\begin{array}{c}
\tilde{p}\\
\tilde{v}
\end{array}\right)\coloneqq\left(\begin{array}{cc}
1 & v_{0}e_{3}^{\top}\\
v_{0}e_{3} & 1
\end{array}\right)\left(\begin{array}{c}
p\\
v
\end{array}\right)
\]
as new unknowns yielding the block system
\[
\left(\partial_{0}M_{0}+\left(\begin{array}{cc}
0 & \dive\\
\grad & 0
\end{array}\right)\right)\left(\begin{array}{c}
\tilde{p}\\
\tilde{v}
\end{array}\right)=\left(\begin{array}{c}
f\\
0
\end{array}\right)
\]
with
\[
M_{0}=\left(\begin{array}{cc}
\frac{1}{1-v_{0}^{2}} & -\frac{v_{0}}{1-v_{0}^{2}}e_{3}^{\top}\\
-\frac{v_{0}}{1-v_{0}^{2}}e_{3} & \frac{1}{1-v_{0}^{2}}
\end{array}\right).
\]
Since for well-posedness of this system we should have that $M_{0}$
is strictly positive definite, we obtain
\begin{equation}
v_{0}^{2}<1,\label{eq:Mach-less}
\end{equation}
i.e. Mach number less than $1$, as a reasonable constraint. This
constraint also recurs in the perspective of a second order approach.
Indeed, eliminating\footnote{Eliminating instead the pressure $p$ yields the Galbrun equation,
\cite{Galbrun1931}, 
\[
\partial_{0}^{2}v+2v_{0}\partial_{0}\partial_{3}v-\left(\grad\dive-v_{0}^{2}\partial_{3}^{2}\right)v=-\grad f
\]
for the velocity field $v$, which seemingly also suggests to impose
\eqref{eq:Mach-less}, compare e.g. \cite{BONNETBENDHIA2001}. } the velocity field $v$ from the system \eqref{eq:Euler}, \eqref{eq:Cont}
yields
\[
\partial_{0}^{2}p+2v_{0}\partial_{0}\partial_{3}p-\left(\partial_{1}^{2}+\partial_{2}^{2}+\left(1-v_{0}^{2}\right)\partial_{3}^{3}\right)p=\partial_{0}f
\]
and requiring ellipticity of the second order spatial operator $\left(\partial_{1}^{2}+\partial_{2}^{2}+\left(1-v_{0}^{2}\right)\partial_{3}^{3}\right)$
imposes again \eqref{eq:Mach-less}. In contrast, looking at the original
system as a standard evolution equation in $L^{2}\left(\mathbb{R}^{3},\mathbb{R}^{4}\right)$
we see that 
\begin{equation}
\left(\partial_{0}+\left(\begin{array}{cccc}
v_{0}\partial_{3} & \partial_{1} & \partial_{2} & \partial_{3}\\
\partial_{1} & v_{0}\partial_{3} & 0 & 0\\
\partial_{2} & 0 & v_{0}\partial_{3} & 0\\
\partial_{3} & 0 & 0 & v_{0}\partial_{3}
\end{array}\right)\right)\left(\begin{array}{c}
p\\
v_{1}\\
v_{2}\\
v_{3}
\end{array}\right)=\left(\begin{array}{c}
f\\
0\\
0\\
0
\end{array}\right)\label{eq:Friedrichs0}
\end{equation}
results in a well-posed system for arbitrary $v_{0}\in\mathbb{R}$,
since the spatial operator is essentially skew-selfadjoint by the
function calculus of the commuting skew-selfadjoint partial derivatives
provided by the spatial Fourier transform. This exposes the constraint
\eqref{eq:Mach-less} as a mathematical artefact and strongly suggests
a direct approach via \eqref{eq:Friedrichs0} and for the general
anisotropic, inhomogeneous case via Friedrichs type systems. This
has been successfully done in \cite{hgg2019wellposedness} in the
framework of \cite{Rauch1985} for the Euclidean case with sufficiently
smooth boundary and material properties. The purpose of this paper
is to explore the Friedrichs type approach from a more functional
analytical perspective with the aim of including in particular non-smooth
boundaries. Indeed, no boundary regularity is needed, which is a benefit
of the carefully constructed functional analytical setting. In particular,
the common-place use of boundary traces is avoided. To include a variety
of geometries we found it also helpful to give the discussion a more
differential geometric flavor by generalising the discussion to a
differential form setting on Riemannian manifolds. This allows to
conveniently handle coordinate transformations since the differential
form calculus on manifolds provides a machinery for this. As a benefit
we also cover with our approach wave propagation for example on surfaces
(with or without boundary). Since there is no added mathematical difficulty
we include as a by-product forms of arbitrary degree, thus addressing
for example also Maxwell's equations\footnote{The latter is of course a mere historical comment, since Maxwell's
equations in moving media are properly considered in the frame work
of relativity theory, which essentially removes the difficulty of
having to deal with a separate drift term.}in a unified setting. More precisely, we consider equations of the
form 
\begin{equation}
\left(\partial_{0}M_{0}+\alpha\left(\begin{array}{cc}
\nabla_{X_{0}} & 0\\
0 & \nabla_{X_{0}}
\end{array}\right)M_{0}+M_{1}+\left(\begin{array}{cc}
0 & -d^{\ast}\\
d & 0
\end{array}\right)\right)U=F\label{eq:abstatc_prob}
\end{equation}
on a non-empty open subset of a Riemannian manifold. Here, $d$ denotes
the exterior derivative and $\nabla_{X_{0}}$ the covariant derivative
in direction of a suitable vector field $X_{0}$. The term $\left(\begin{array}{cc}
0 & -d^{\ast}\\
d & 0
\end{array}\right)$ generalises the operator $\left(\begin{array}{cc}
0 & \dive\\
\grad & 0
\end{array}\right)$ in the euclidean situation and $\alpha\left(\begin{array}{cc}
\nabla_{X_{0}} & 0\\
0 & \nabla_{X_{0}}
\end{array}\right)$ is the generalisation of the concrete drift term in direction $e_{3}$
from above. The operators $M_{0}$ and $M_{1}$ are assumed to be
bounded and incorporate material parameters such as mass density,
conductivity, etc. in applications. 

The main goal is to prove the well-posedness of problem \eqref{eq:abstatc_prob}
in a suitable sense. The main idea to tackle this problem is to replace
the drift operator by another drift operator, which is given in terms
of the Lie-derivative of $X_{0}$, to decompose this new drift operator
in its skew-selfadjoint and selfadjoint part and to show that the
sum of the skew-selfadjoint part and the spatial operator $\left(\begin{array}{cc}
0 & -d^{\ast}\\
d & 0
\end{array}\right)$ is essentially skew-selfadjoint. Thus, the problem can be rewritten
as an equation of the form 
\[
\left(\partial_{0}M_{0}+\tilde{M}_{1}+A\right)U=F,
\]
where $A$ is skew-selfadjoint, $M_{0}$ and $\tilde{M}_{1}$ are
bounded, such that $M_{0}$ is selfadjoint and strictly positive definite.
Then it is easy to see, that the problem is well-posed by invoking
the theory of $C_{0}$-semigroups (note that $-\sqrt{M_{0}^{-1}}(\tilde{M}_{1}+A)\sqrt{M_{0}^{-1}}$
generates a $C_{0}$-semigroup) or by the theory of evolutionary equations,
which will be introduced in the next section. The crucial part in
the proof is to show the essential skew-selfadjointness mentioned
above. For doing so, commutator, or more generally, transmutator relations
between different differentiation operators are needed. Hence, we
will provide some abstract results on commutators and transmutators
of operators in Hilbert spaces in Section \ref{sec:Sums-of-Skew-Selfadjoint},
which may be useful also for other applications. Finally, in Section
\ref{sec:An-Application-to} we deal with problem \eqref{eq:abstatc_prob}.
First we introduce all differential operators on Riemannian manifolds
needed in the forthcoming subsections. Then in Subsection \ref{subsec:The-Equations-on}
we prove our main result Theorem \ref{thm:main} showing the well-posedness
of (\ref{eq:abstatc_prob}) in suitable sense under certain restrictions
on the vector field $X_{0}$. Moreover, we show that in cylindrical
domains $\Sigma\times\R$ or $\Sigma\times]-1/2,1/2[$ for $X_{0}=e_{3}$
the assumptions on the vector field are satisfied, so that our solution
theory applies in these cases. Moreover, we comment on how the solution
theory carries over to isometrically transformed manifolds, allowing
to deal with deformed pipes, etc. In the last subsection we provide
an abstract localisation technique, which allows to ``glue together''
different different open subsets $\Omega_{1},\Omega_{2}$ of a manifold,
so that the solution theory on each part carries over to their union
$\Omega_{1}\cup\Omega_{2}$. %

\section{Some Hilbert Space Solution Theory\label{sec:Some-Hilbert-Space}}

We recall the basic Hilbert space setting for dealing with \emph{evolutionary
}equations; that is, differential equations of the form 
\[
\left(\partial_{0}M_{0}+M_{1}+A\right)U=F.
\]
The results are based on the observations made in \cite{Picard2009}
(see also \cite[Chapter 6]{PicardMcGhee2011} and \cite{seifert2020evolutionary}).
Throughout, let $H$ be a real Hilbert space.
\begin{defn*}
For $\rho\geq0$ we define the space 
\[
L_{2,\rho}(\R;H)\coloneqq\{f:\R\to H\,;\,f\text{ measurable, }\int_{\R}\|f(t)\|^{2}\e^{-2\rho t}\d t\}
\]
equipped with the obvious inner product. Moreover, we define the weighted
Sobolev space
\[
H_{\rho}^{1}(\R;H)\coloneqq\{f\in L_{2,\rho}(\R;H)\,;\,f'\in L_{2,\rho}(\R;H)\},
\]
where $f'$ is meant in the sense of distributions, and the operator
\[
\partial_{0,\rho}:H_{\rho}^{1}(\R;H)\subseteq L_{2,\rho}(\R;H)\to L_{2,\rho}(\R;H),\quad f\mapsto f'.
\]
\end{defn*}
\begin{prop}
The operator $\partial_{0,\rho}$ is normal with $\sym\partial_{0,\rho}=\frac{1}{2}\left(\overline{\partial_{0,\rho}+\partial_{0,\rho}^{\ast}}\right)=\rho.$
Moreover, $\partial_{0,\rho}$ is invertible if and only if $\rho>0$
and in this case 
\[
\left(\partial_{0,\rho}^{-1}f\right)(t)=\int_{-\infty}^{t}f(s)\d s\quad(t\in\R,f\in L_{2,\rho}(\R;H)).
\]
\end{prop}

If the choice of $\rho$ is clear from the context, we drop the additional
index and just write $\partial_{0}$.
\begin{thm}[{\cite[Solution Theory]{Picard2009}}]
\label{thm:sol_theory} Let $M_{0},M_{1}\in L(H)$ with $M_{0}=M_{0}^{\ast}$
and $A:\dom(A)\subseteq H\to H$ a skew-selfadjoint operator. We extend
all these operators to $L_{2,\rho}(\R;H)$ in the canonical way. Moreover,
we assume that there exists $\rho_{0}\geq0$ and $c>0$ such that
\[
\langle\left(\rho M_{0}+\frac{1}{2}(M_{1}+M_{1}^{\ast})\right)x,x\rangle\geq c\|x\|^{2}\quad(x\in H)
\]
for all $\rho\geq\rho_{0}.$ Then $\partial_{0}M_{0}+M_{1}+A$ is
closable in $L_{2,\rho}(\R;H)$ for each $\rho\geq\rho_{0}$ and the
closure is continuously invertible with 
\[
\|\left(\overline{\partial_{0}M_{0}+M_{1}+A}\right)^{-1}\|_{L(L_{2,\rho}(\R;H))}\leq\frac{1}{c}.
\]
Moreover, the operator $S_{\rho}\coloneqq\left(\overline{\partial_{0}M_{0}+M_{1}+A}\right)^{-1}$
is causal; i.e., for all $f\in L_{2,\rho}(\R;H)$ such that $\spt f\subseteq\R_{\geq a}$
for some $a\in\R$ it follows that $\spt S_{\rho}f\subseteq\R_{\geq a}$,
and the operator $S_{\rho}$ is independent of the choice of $\rho\geq\rho_{0}$
in the sense that $S_{\rho}f=S_{\mu}f$ for each $f\in L_{2,\rho}(\R;H)\cap L_{2,\mu}(\R;H)$
and $\mu,\rho\geq\rho_{0}$.
\end{thm}

\begin{rem}
$\:$

\begin{enumerate}[(a)]

\item The latter theorem shows that the problem of finding $u\in L_{2,\rho}(\R;H)$
such that 
\[
\left(\overline{\partial_{0}M_{0}+M_{1}+A}\right)u=f
\]
for some $f\in L_{2,\rho}(\R;H)$ is well-posed in the sense of Hadamard.
Indeed, the bijectivity of $\left(\overline{\partial_{0}M_{0}+M_{1}+A}\right)$
yields the existence and uniqueness of a solution for each $f\in L_{2,\rho}(\R;H)$,
while the continuity of the inverse shows the continuous dependence
of the solution $u$ on the data $f$. Moreover, the causality shows
that the equation models a physically reasonable process in time. 

\item The latter theorem is just a special case of \cite[Solution Theory]{Picard2009},
where equations of the form 
\[
(\partial_{0}M(\partial_{0}^{-1})+A)u=f
\]
 for a suitable operator-valued function $M$ of $\partial_{0}^{-1}$
are considered. The special case in Theorem \prettyref{thm:sol_theory}
corresponds to the choice $M(\partial_{0}^{-1})=M_{0}+\partial_{0}^{-1}M_{1}.$
Moreover, several generalisations of Theorem \prettyref{thm:sol_theory}
can be found in the literature, for instance allowing to treat non-autonomous
problems (\cite{PTW2013_nonauto,Waurick2015}) or non-linear problems
(\cite{Trostorff2012,Trostorff2020}).

\end{enumerate}

As an immediate consequence of Theorem \ref{thm:sol_theory}, we obtain
the following perturbation result.
\end{rem}

\begin{cor}
\label{cor:perturbation} Let $M_{0},M_{1}\in L(H)$ with $M_{0}=M_{0}^{\ast}$
and $A:\dom(A)\subseteq H\to H$ a skew-selfadjoint operator. Moreover,
assume that there exists $c>0$ such that 
\[
\langle M_{0}x,x\rangle\geq c\|x\|^{2}\quad(x\in H).
\]
Then there exists $\rho_{0}\geq0$ such that for all $\rho\geq\rho_{0}$
the operator $\partial_{0}M_{0}+M_{1}+A$ is closable and continuously
invertible in $L_{2,\rho}(\R;H)$.
\end{cor}

\begin{proof}
We choose $\rho_{0}$ such that $\tilde{c}\coloneqq\rho_{0}c-\|M_{1}\|>0$
. Then for $\rho\geq\rho_{0}$ we have that 
\[
\langle\rho M_{0}+\frac{1}{2}(M_{1}+M_{1}^{\ast})x,x\rangle\geq\tilde{c}\|x\|^{2}
\]
and hence, the claim follows from Theorem \ref{thm:sol_theory}.
\end{proof}

\section{Sums of Skew-Selfadjoint Operators: Weak=Strong\label{sec:Sums-of-Skew-Selfadjoint}}

There is a well-developed general theory of sums of discontinuous
operators in Banach spaces, see \cite{daPrato1975}. For sake of simplicity
and transparency, however, we chose instead a more ``pedestrian''
approach fitted to the Hilbert space setting and in keeping with Friedrichs
original approach to positive symmetric systems, \cite{Friedrichs1958},
in which the classical question of the relation between weak and strong
extensions are of significance, \cite{Friedrichs1944}.

\subsection{Transmutators and Commutators}

We begin with defining the notions of transmutators and commutators
for operators on Hilbert spaces. For doing so, let $H_{0},H_{1}$
be Hilbert spaces.
\begin{defn*}
Let $L:H_{1}\to H_{1}$ and $R:H_{0}\to H_{0}$ be continuous linear
operators and $C:\dom(C)\subseteq H_{0}\to H_{1}$ a densely defined
closed linear operator such that $R[\dom(C)]\subseteq\dom(C).$ We
define
\[
\left[L,C,R\right]\coloneqq LC-CR:\dom(C)\subseteq H_{0}\to H_{1}
\]
the \emph{transmutator of $L,R$ and $C$. }Moreover, if $H_{0}=H_{1}$
we set 
\[
\left[R,C\right]\coloneqq\left[R,C,R\right]
\]
the \emph{commutator of $R$ and $C$ }and for convenience
\[
\left[C,R\right]\coloneqq-\left[R,C\right].
\]
\end{defn*}

\begin{lem}
\label{lem:transmutation}Let $L:H_{1}\to H_{1}$ and $R:H_{0}\to H_{0}$
be continuous linear operators and $C:\dom(C)\subseteq H_{0}\to H_{1}$
a densely defined closed linear operator such that $R[\dom(C)]\subseteq\dom(C).$
Assume that $[L,C,R]$ is continuous. 

\begin{enumerate}[(a)]

\item Then $LC$ is closable with 
\[
\overline{LC}\subseteq CR+\overline{[L,C,R]}.
\]
\item If, additionally, $\dom(C)$ is a core for $CR$, then 
\[
\overline{LC}=CR+\overline{[L,C,R]}
\]
and hence, in particular 
\[
\dom(\overline{LC})=\dom(CR).
\]
\end{enumerate}
\end{lem}

\begin{proof}
We note that $\overline{[L,C,R]}:H_{0}\to H_{1}$ is continuous, since
$[L,C,R]$ is densely defined and continuous by assumption. Moreover,
since $C$ is closed, so is $CR$. Since clearly 
\[
LC\subseteq CR+\overline{[L,C,R]}
\]
we infer (a) holds. For showing (b), assume now that $\dom(C)$ is
a core for $CR$ and let $x\in\dom(CR).$ Then we find a sequence
$(x_{n})_{n\in\mathbb{N}}$ in $\dom(C)$ with $x_{n}\to x$ and $CRx_{n}\to CRx$
as $n\to\infty.$ Thus, we have 
\[
LCx_{n}=CRx_{n}+\overline{[L,C,R]}x_{n}\to CRx+\overline{[L,C,R]}x\quad(n\to\infty)
\]
due to the continuity of $\overline{[L,C,R]}.$ The latter proves
$x\in\dom(\overline{LC})$, which shows the claim. 
\end{proof}
\begin{defn*}
Let $C:\dom(C)\subseteq H_{0}\to H_{1}$ and $D:\dom(D)\subseteq H_{0}\to H_{1}$
be densely defined closed linear operators. We call $C$ and $D$
\emph{essentially equal},\emph{ }if $\dom(C)=\dom(D)$ and $C-D$
is continuous.
\end{defn*}
\begin{rem}
~

\begin{enumerate}[(a)]

\item If $C$ and $D$ are essentially equal, we infer that $C=D+\overline{C-D}$. 

\item In the situation of Lemma \ref{lem:transmutation} (b) we have
that $\overline{LC}$ and $CR$ are essentially equal. 

\end{enumerate} 
\end{rem}

\subsection{Commutators with resolvents of $m$-accretive Operators and Convergence
results.}

We now focus on the case $H\coloneqq H_{0}=H_{1}$ and $L=R$. A class
of operators for which we will apply Lemma \ref{lem:transmutation}
is the class of $m$-accretive operators. For doing so, we recall
the definition of these operators. 
\begin{defn*}
Let $C:\dom(C)\subseteq H\to H$ be a linear densely defined closed
operator. Then $C$ is called \emph{accretive}, if 
\[
\forall x\in\dom(C):\:\langle x,Cx\rangle\geq0.
\]
Moreover, $C$ is called $m$-\emph{accretive}, if $C$ is accretive
and $1+C$ is onto\footnote{Equivalently, if $C$ and $C^{*}$ are accretive.}.
\end{defn*}
\begin{rem}
$\,$

\begin{enumerate}[(a)]

\item We remark that a linear operator $C:\dom(C)\subseteq H\to H$
is $m$-accretive if and only if $(1+\eta C)^{-1}:H\to H$ is continuous
for all $\eta\geq0$ and 
\[
\|(1+\eta C)^{-1}\|\leq1.
\]
In particular, the closedness and the dense domain of $C$ follow
from this uniform bound of the resolvents. 

\item From (a) we see that $C^{\ast}$ is $m$-accretive if $C$
is $m$-accretive.

\end{enumerate}
\end{rem}

\begin{defn*}
Let $C:\dom(C)\subseteq H\to H$ be a linear densely defined closed
operator. We call $C$ \emph{quasi-$m$-accretive}, if there exists
an $\eta_{0}>0$ such that $(1+\eta C)^{-1}\in L(H)$ for each $0\leq\eta\leq\eta_{0}$
and 
\[
\sup_{0\leq\eta\leq\eta_{0}}\|(1+\eta C)^{-1}\|<\infty.
\]
\end{defn*}
\begin{lem}
\label{lem:resolvent_m_accretive}If $C$ is quasi-$m$-accretive,
then $(1+\eta C)^{-1}\to1$ strongly as $\eta\to0+$.
\end{lem}

\begin{proof}
Since the resolvents are uniformly bounded near zero, it suffices
to prove the convergence for elements in a dense set, say $\dom(C).$
For $x\in\dom(C)$ we compute 
\[
(1+\eta C)^{-1}x-x=-\eta(1+\eta C)^{-1}\left(Cx\right)\to0\quad(\eta\to0+),
\]
where we again have used the uniform boundedness of the resolvents
near zero.
\end{proof}
\begin{lem}
\label{lem:quasi-m-accretive_ex}Let $C:\dom(C)\subseteq H\to H$
be linear densely defined and closed. Moreover, let $B\in L(H)$ such
that $C-B$ is $m$-accretive. Then $C$ is quasi-$m$-accretive.
\end{lem}

\begin{proof}
We choose $\eta_{0}>0$ such that $\eta_{0}\|B\|<1.$ Then we estimate
for each $0\leq\eta\leq\eta_{0}$ 
\begin{align*}
\langle(1+\eta C)x,x\rangle & =\|x\|^{2}+\eta\langle(C-B)x,x\rangle+\eta\langle Bx,x\rangle\\
 & \geq\|x\|^{2}-\eta\|B\|\|x\|^{2}\\
 & \geq(1-\eta_{0}\|B\|)\|x\|^{2},
\end{align*}
which shows that $(1+\eta C)$ is injective. For showing that $1+\eta C$
is onto, we take $y\in H$. By the contraction mapping theorem, there
exists $x\in H$ such that 
\[
x=(1+\eta(C-B))^{-1}(y-\eta Bx)
\]
and it is immediate, that this $x$ satisfies 
\[
(1+\eta C)x=y.
\]
Finally, the estimate above shows that 
\[
\sup_{0\le\eta\leq\eta_{0}}\|(1+\eta C)^{-1}\|\leq\frac{1}{1-\eta_{0}\|B\|}
\]
and hence, the assertion follows.
\end{proof}
\begin{lem}
\label{lem:commutator-resolvent}Let $C:\dom(C)\subseteq H\to H$
be quasi-$m$-accretive and $\alpha:H\to H$ continuous with $\alpha[\dom(C)]\subseteq\dom(C).$
Moreover, we assume that $[\alpha,C]$ is continuous. Then 
\[
[(1+\eta C)^{-1},\alpha]=\eta(1+\eta C)^{-1}\overline{[\alpha,C]}(1+\eta C)^{-1}
\]
for sufficiently small $\eta\geq0.$ In particular 
\[
[(1+\eta C)^{-1},\alpha]\to0\quad(\eta\to0+)
\]
in operator norm. 
\end{lem}

\begin{proof}
Since 
\[
\alpha C\subseteq C\alpha+\overline{[\alpha,C]}
\]

we infer that 
\[
\alpha(1+\eta C)\subseteq(1+\eta C)\alpha+\eta\overline{[\alpha,C]}\quad(\eta\geq0).
\]
Hence, 
\[
(1+\eta C)^{-1}\alpha\subseteq\alpha(1+\eta C)^{-1}+\eta(1+\eta C)^{-1}\overline{[\alpha,C]}(1+\eta C)^{-1}
\]
for a $\eta\geq0$ sufficiently small. Since both sides are continuous
operators on $H$, we derive that 
\[
[(1+\eta C)^{-1},\alpha]=(1+\eta C)^{-1}\alpha-\alpha(1+\eta C)^{-1}=\eta(1+\eta C)^{-1}\overline{[\alpha,C]}(1+\eta C)^{-1}.
\]
Finally, since $(1+\eta C)^{-1}$ is uniformly bounded in $\eta,$
we infer the asserted convergence result. 
\end{proof}
We can prove an even stronger convergence result, than the one in
the previous lemma.
\begin{prop}
\label{prop:C_commuatotr}Let $C:\dom(C)\subseteq H\to H$ be quasi-$m$-accretive
and $\alpha:H\to H$ continuous with $\alpha[\dom(C)]\subseteq\dom(C).$
Moreover, we assume that $[\alpha,C]$ is continuous. Then $C[(1+\eta C)^{-1},\alpha]:H\to H$
is continuous and 
\[
C[(1+\eta C)^{-1},\alpha]\to0
\]
strongly as $\eta\to0+$.
\end{prop}

\begin{proof}
By Lemma \ref{lem:commutator-resolvent} we have that 
\[
[(1+\eta C)^{-1},\alpha]=\eta(1+\eta C)^{-1}\overline{[\alpha,C]}(1+\eta C)^{-1}
\]
and thus 
\begin{align*}
C[(1+\eta C)^{-1},\alpha] & =\eta C(1+\eta C)^{-1}\overline{[\alpha,C]}(1+\eta C)^{-1}\\
 & =\overline{[\alpha,C]}(1+\eta C)^{-1}-(1+\eta C)^{-1}\overline{[\alpha,C]}(1+\eta C)^{-1}
\end{align*}
for each sufficiently small $\eta\geq0.$ The latter expression yields
the claim as $(1+\eta C)^{-1}\to1$ strongly as $\eta\to0+$ by \prettyref{lem:resolvent_m_accretive}. 
\end{proof}
With these preparations at hand, we can state our main theorem of
this subsection.
\begin{thm}
\label{thm:weak-strong-accretive}Let $C:\dom(C)\subseteq H\to H$
be quasi-$m$-accretive and $\alpha:H\to H$ continuous with $\alpha[\dom(C)]\subseteq\dom(C).$
Moreover, we assume that $[\alpha,C]$ is continuous. Then 
\[
\overline{\alpha C}=C\alpha+\overline{[\alpha,C]}.
\]
\end{thm}

\begin{proof}
By Lemma \ref{lem:transmutation} it suffices to prove that $\dom(C)$
is a core for $C\alpha.$ For doing so, let $x\in\dom(C\alpha)$ and
define 
\[
x_{n}\coloneqq\left(1+\frac{1}{n}C\right)^{-1}x\quad(n\in\mathbb{N}\text{ sufficiently large}).
\]
Then $x_{n}\in\dom(C)$ for each $n\in\mathbb{N}$ and $x_{n}\to x$
as $n\to\infty$ by \prettyref{lem:resolvent_m_accretive}. Moreover,
\begin{align*}
C\alpha x_{n} & =C\alpha\left(1+\frac{1}{n}C\right)^{-1}x\\
 & =C\left(1+\frac{1}{n}C\right)^{-1}\alpha x-C\left[\left(1+\frac{1}{n}C\right)^{-1},\alpha\right]x\\
 & =\left(1+\frac{1}{n}C\right)^{-1}C\alpha x-C\left[\left(1+\frac{1}{n}C\right)^{-1},\alpha\right]x\\
 & \to C\alpha x\quad(n\to\infty)
\end{align*}
by Proposition \ref{lem:commutator-resolvent}, which yields the claim. 
\end{proof}
\begin{rem}
The latter theorem can be interpreted as a ``weak=strong'' result
in the following sense:
\begin{itemize}
\item $\left(C^{*}\right)^{*}=C^{**}=\overline{C}$ for a closable densely
defined linear operator $C$ (``weak=strong'' for operators). This
yields in particular $\overline{\alpha C}=\left(C^{*}\alpha^{*}\right)^{*}$.
\item With our commutator results we get $\overline{\alpha C}=C\alpha+\overline{\left[\alpha,C\right]}=\left(\alpha^{*}C^{*}\right)^{*}+\overline{\left[\alpha,C\right]}$
showing that $\overline{\alpha C}$ and $\left(\alpha^{\ast}C^{\ast}\right)^{\ast}$
are essentially equal (``weak=strong'' for operator products).
\end{itemize}
\end{rem}

We conclude this subsection by the following corollary.
\begin{cor}
\label{cor:adjoint_m_accretive}Let $C:\dom(C)\subseteq H\to H$ linear
closed and densely defined, and $\alpha:H\to H$ linear and continuous.
Assume that $\alpha[\dom(C)]\subseteq\dom(C)$ and $[\alpha,C]$ is
continuous. Then 
\[
[\alpha,C]^{\ast}=\overline{[C^{\ast},\alpha^{\ast}]}.
\]
If, additionally, $C$ is quasi-$m$-accretive, we have that
\[
(\alpha C)^{\ast}=\overline{\alpha^{\ast}C^{\ast}}+\overline{\left[C^{*},\alpha^{*}\right]}.
\]
\end{cor}

\begin{proof}
Since we have $\alpha^{\ast}C^{\ast}\subseteq\left(C\alpha\right)^{\ast}$,
it follows that
\[
\alpha^{\ast}C^{\ast}+[\alpha,C]^{\ast}\subseteq(C\alpha)^{\ast}+[\alpha,C]^{\ast}\subseteq\left(C\alpha+[\alpha,C]\right)^{\ast}=(\alpha C)^{\ast}=C^{\ast}\alpha^{\ast},
\]
which proves that 
\[
[\alpha,C]^{\ast}x=[C^{\ast},\alpha^{\ast}]x\quad(x\in\dom(C^{\ast})).
\]
Since $\dom(C^{\ast})$ is dense and $[\alpha,C]^{\ast}$ is bounded,
we derive 
\[
[\alpha,C]^{\ast}=\overline{[C^{\ast},\alpha^{\ast}]}.
\]

For the second statement, we recall that 
\[
\overline{\alpha C}=C\alpha+\overline{[\alpha,C]}
\]
by Theorem \ref{thm:weak-strong-accretive}. Taking adjoints on both
sides and using the result above, we obtain 
\[
(\alpha C)^{\ast}=(C\alpha)^{\ast}+\overline{[C^{\ast},\alpha^{\ast}]}.
\]
Since 
\[
\overline{\alpha^{\ast}C^{\ast}}=(\alpha^{\ast}C^{\ast})^{\ast\ast}=(C\alpha)^{\ast},
\]
the assertion follows.
\end{proof}

\subsection{Commutators for Quasi-Skew-Selfadjoint Operators\label{subsec:Commutators-for-Quasi-Skew-Selfa}}

Throughout, let $H$ be a Hilbert space, $\alpha\in L(H)$ and $C:\dom(C)\subseteq H\to H$
linear closed and densely defined.
\begin{defn*}
We call $C$ \emph{quasi-skew-selfadjoint}, if $\dom(C)=\dom(C^{\ast})$
and $\sym C\coloneqq\frac{1}{2}\overline{\left(C+C^{\ast}\right)}$
is bounded. 
\end{defn*}
\begin{rem}
$\:$\begin{enumerate}[(a)] \item If $C$ is skew-selfadjoint, then
$C$ is quasi-skew-selfadjoint, since in this case $\sym C=0.$

\item If $C$ is quasi-skew-selfadjoint, we have 
\[
C=-C^{\ast}+2\sym C
\]
and thus, 
\begin{align*}
C-\sym C & =-\left(C^{\ast}-\sym C\right)\\
 & =-\left(C-\sym C\right)^{\ast},
\end{align*}
since $\sym C$ is selfadjoint and bounded, and hence, $C-\sym C$
is skew-selfadjoint and thus, $C$ is quasi-$m$-accretive by Lemma
\ref{lem:quasi-m-accretive_ex} (note that skew-selfadjoint operators
are $m$-accretive).

\end{enumerate}
\end{rem}

\begin{thm}
\label{thm:quasi-skewselfadjoint}Let $C$ be quasi-skew-selfadjoint
and $\alpha\in L(H)$ selfadjoint. Moreover, assume that $\alpha[\dom(C)]\subseteq\dom(C)$
and $[\alpha,C]$ is continuous. Then 
\[
\skew(\alpha C)\coloneqq\frac{1}{2}\overline{\alpha C-(\alpha C)^{\ast}}
\]
is skew-selfadjoint.
\end{thm}

\begin{proof}
Using Lemma \ref{cor:adjoint_m_accretive} (note that $C$ is quasi-$m$-accretive)
we compute for $x\in\dom(C)=\dom(C^{\ast})$
\begin{align*}
\frac{1}{2}\left(\alpha C-\left(\alpha C\right)^{*}\right)x & =\frac{1}{2}\left(\alpha Cx-\alpha C^{*}x-\left[C^{*},\alpha\right]x\right)\\
 & =\alpha Cx-\frac{1}{2}\alpha(C+C^{\ast})x-\frac{1}{2}[C^{\ast},\alpha]x\\
 & =\alpha Cx-\alpha\sym(C)x-\frac{1}{2}[C^{\ast},\alpha]x,
\end{align*}
which yields 
\begin{equation}
\skew(\alpha C)=\overline{\alpha C}-\alpha\sym(C)-\frac{1}{2}\overline{[C^{\ast},\alpha]}.\label{eq:skew}
\end{equation}
Thus, 
\begin{align*}
\skew(\alpha C)^{\ast} & =(\alpha C)^{\ast}-\sym(C)\alpha-\frac{1}{2}\left[C^{\ast},\alpha\right]^{\ast}\\
 & =(C^{\ast}-\sym(C))\alpha-\frac{1}{2}\overline{[\alpha,C]}
\end{align*}
again by Lemma \ref{cor:adjoint_m_accretive}. Since $C$ is quasi-skew-selfadjoint,
we obtain 
\[
C^{\ast}-\sym(C)=-\left(C-\sym(C)\right)
\]
and hence, 
\[
\skew(\alpha C)^{\ast}=-\left(\left(C-\sym(C)\right)\alpha+\frac{1}{2}\overline{[\alpha,C]}\right).
\]
Hence, it suffices to prove that 
\[
\left(C-\sym(C)\right)\alpha+\frac{1}{2}\overline{[\alpha,C]}=\skew(\alpha C).
\]

By Theorem \ref{thm:weak-strong-accretive} we get that 
\[
C\alpha+\overline{[\alpha,C]}=\overline{\alpha C}.
\]
Hence, using formula (\ref{eq:skew}) we need to show that 
\[
-\sym(C)\alpha-\frac{1}{2}\overline{[\alpha,C]}=-\alpha\sym(C)-\frac{1}{2}\overline{[C^{\ast},\alpha]}.
\]
Since the operators on both sides are all bounded, it suffices to
check this equality on a dense set, say $\dom(C).$ However, on $\dom(C)$
we have 
\begin{align*}
\sym(C)\alpha+\frac{1}{2}\overline{[\alpha,C]} & =\frac{1}{2}(C\alpha+C^{\ast}\alpha)+\frac{1}{2}(\alpha C-C\alpha)\\
 & =\frac{1}{2}\alpha C+\frac{1}{2}C^{\ast}\alpha\\
 & =\alpha\sym(C)+\frac{1}{2}(C^{\ast}\alpha-\alpha C^{\ast})\\
 & =\alpha\sym(C)+\frac{1}{2}\overline{[C^{\ast},\alpha]},
\end{align*}
which shows the claim.
\end{proof}

\subsection{Transmutators and Sums of Operators}
\begin{lem}
\label{lem:adjoint_transmut}Let $C:\dom(C)\subseteq H_{0}\to H_{1}$
and $D:\dom(D)\subseteq H_{0}\to H_{1}$ two densely defined closed
linear operators such that $\dom(C)\cap\dom(D)$ is dense in $H_{0}$.
Let $R\in L(H_{0}),L\in L(H_{1})$ with $R[\dom(C)\cap\dom(D)]\subseteq\dom(C)\cap\dom(D)$
and $[L,C+D,R]$ is continuous. Then $L^{\ast}[\dom((C+D)^{\ast})]\subseteq\dom((C+D)^{\ast})$
and
\[
[R^{\ast},(C+D)^{\ast},L^{\ast}]\subseteq-[L,C+D,R]^{\ast}.
\]
Moreover, if $\dom((C+D)^{\ast})$ is dense, we obtain 
\[
\overline{[R^{\ast},(C+D)^{\ast},L^{\ast}]}=-[L,C+D,R]^{\ast}.
\]
\end{lem}

\begin{proof}
Let $x\in\dom((C+D)^{\ast})$ and $y\in\dom(C+D)=\dom(C)\cap\dom(D).$
Then we compute 
\begin{align*}
\langle(C+D)y,L^{\ast}x\rangle & =\langle L(C+D)y,x\rangle\\
 & =\langle[L,C+D,R]y+(C+D)Ry,x\rangle\\
 & =\langle y,[L,C+D,R]^{\ast}x+R^{\ast}(C+D)^{\ast}x\rangle,
\end{align*}
which shows $L^{\ast}x\in\dom((C+D)^{\ast})$ and 
\[
(C+D)^{\ast}L^{\ast}x=[L,C+D,R]^{\ast}x+R^{\ast}(C+D)^{\ast}x.
\]
In other words 
\[
[R^{\ast},(C+D)^{\ast},L^{\ast}]\subseteq-[L,C+D,R]^{\ast}.
\]
If additionally, $\dom((C+D)^{\ast})$ is dense the continuity of
the right-hand side yields
\[
\overline{[R^{\ast},(C+D)^{\ast},L^{\ast}]}=-[L,C+D,R]^{\ast}.\tag*{\qedhere}
\]
 
\end{proof}
\begin{thm}
\label{thm:sum}Let $C:\dom(C)\subseteq H_{0}\to H_{1}$ and $D:\dom(D)\subseteq H_{0}\to H_{1}$
two densely defined closed linear operators such that $\dom(C)\cap\dom(D)$
is dense in $H_{0}$. Moreover, let $L_{n}\in L(H_{1}),R_{n}\in L(H_{0})$
such that $R_{n}[\dom(C)\cap\dom(D)]\subseteq\dom(C)\cap\dom(D)$
and 
\[
[L_{n},C+D,R_{n}]
\]
is continuous for each $n\in\N$. If $\left(C+D\right)^{*}L_{n}^{*}=\left(C^{*}+D^{*}\right)L_{n}^{*}$
for each $n\in\N$ and $R_{n}^{*}\overset{s}{\to}1$, $L_{n}^{*}\overset{s}{\to}1$
as well as 
\[
\left[L_{n},C+D,R_{n}\right]^{*}\overset{s}{\to}0,
\]
as $n\to\infty,$ we obtain
\[
\overline{C^{*}+D^{*}}=\left(C+D\right)^{*}.
\]
\end{thm}

\begin{proof}
Since always $C^{\ast}+D^{\ast}\subseteq(C+D)^{\ast}$ it follows
that $\overline{C^{\ast}+D^{\ast}}\subseteq(C+D)^{\ast}.$ To prove
the remaining inclusion let $v\in\dom((C+D)^{\ast})$. Then $L_{n}^{*}v\to v$
as $n\to\infty$ and by Lemma \ref{lem:adjoint_transmut} we obtain
\begin{align*}
\left(C^{*}+D^{*}\right)L_{n}^{*}v & =\left(C+D\right)^{*}L_{n}^{\ast}v\\
 & =R_{n}^{\ast}(C+D)^{\ast}v-[R_{n}^{\ast},(C+D)^{\ast},L_{n}^{\ast}]v\\
 & =R_{n}^{\ast}(C+D)^{\ast}v+[L_{n},(C+D),R_{n}]^{\ast}v\\
 & \to(C+D)^{\ast}v.
\end{align*}
The latter gives $v\in\dom(\overline{C^{\ast}+D^{\ast}})$ and thus,
the assertion follows.
\end{proof}

\section{An Application to Maxwell's and Acoustics Equations with Drift Term\label{sec:An-Application-to}}

\subsection{Lie Derivative and Co-Variant Derivative}

Throughout, let $(M,g)$ be a smooth Riemannian manifold of odd dimension
$n$ and $\Omega\subseteq M$ open. We begin to define the space of
tensor fields on $\Omega$.
\begin{defn*}
Let $k,\ell\in\N$.

\begin{enumerate}[(a)] 

\item Let $p\in\Omega$. By $T_{\ell}^{k}(T_{p}M)$ we denote the
set of all tensors of type $(k,\ell)$ on the tangent space $T_{p}M$,
i.e. 
\[
T_{\ell}^{k}(T_{p}M)=\bigotimes_{j=1}^{k}(T_{p}M)'\otimes\bigotimes_{i=1}^{\ell}T_{p}M.
\]
\item A function $X:\Omega\to\bigcup_{p\in\Omega}T_{p}M$ with $X(p)\in T_{p}M$
for each $p\in\Omega$ is called a \emph{vectorfield on $\Omega$},
if for all $f\in C^{\infty}(\Omega)$ the mapping 
\[
\Omega\ni p\mapsto X(p)(f)\in\R
\]
is smooth. We denote by $T_{1}^{0}(\Omega)$ the set of all vectorfields
on $\Omega.$

\item A function $Y:\Omega\to\bigcup_{p\in\Omega}(T_{p}M)'$ with
$Y(p)\in\left(T_{p}M\right)'$ for each $p\in\Omega$ is called a
\emph{covectorfield on $\Omega$}, if for all $X\in T_{1}^{0}(\Omega)$
the mapping 
\[
\Omega\ni p\mapsto Y(p)(X(p))\in\R
\]
is smooth. We denote by $T_{0}^{1}(\Omega)$ the set of all covectorfields
on $\Omega.$

\item A function $T:\Omega\to\bigcup_{p\in\Omega}T_{\ell}^{k}(T_{p}M)$
with $T(p)\in T_{\ell}^{k}(T_{p}M)$ for each $p\in\Omega$ is called
a \emph{tensorfield of type $(k,\ell)$ on $\Omega$}, if for all
$X_{1},\ldots,X_{k}\in T_{1}^{0}(\Omega)$ and all $Y_{1},\ldots,Y_{\ell}\in T_{0}^{1}(\Omega)$
the mapping 
\[
\Omega\ni p\mapsto T(p)(X_{1}(p),\ldots,X_{k}(p),Y_{1}(p),\ldots,Y_{\ell}(p))\in\R
\]
is smooth. We denote by $T_{\ell}^{k}(\Omega)$ the set of all tensorfields
of type $(k,\ell)$ on $\Omega.$

\item A tensorfield $T$ of type $(k,0)$ on $\Omega$ is called
a \emph{$k$-form on $\Omega$}, if for all $p\in\Omega$ 
\[
T(p)\in\Lambda^{k}(T_{p}M),
\]
where 
\[
\Lambda^{k}(T_{p}M)\coloneqq\{f\in T_{0}^{k}(T_{p}M)\,;\,f\text{ alternating}\}.
\]
The set of $k$-forms on $\Omega$ is denoted by $\Lambda^{k}(\Omega).$ 

\end{enumerate}
\end{defn*}
\begin{rem}
Note that $\Lambda^{k}(\Omega)=\{0\}$ for each $k>n$. 
\end{rem}

Since we want to employ the framework of evolutionary equations, see
\prettyref{sec:Some-Hilbert-Space}, we need to define a suitable
Hilbert space structure. This is done as follows.
\begin{defn*}
Let $k\in\N.$ We define the set $L_{2}^{k}(\Omega)$ as the completion
of 
\[
\{T\in\Lambda^{k}(\Omega)\,;\,\int_{\Omega}\|T(p)\|_{\Lambda^{k}(T_{p}M)}^{2}\d V(p)<\infty\}
\]
with respect to the norm 
\[
T\mapsto\left(\int_{\Omega}\|T(p)\|_{\Lambda^{k}(T_{p}M)}^{2}\d V(p)\right)^{1/2}.
\]
Here $V$ denotes the volume element of the Riemannian manifold $M.$ 
\end{defn*}
\begin{rem}
For $k\in\N$ we have 
\[
L_{2}^{k}(\Omega)=\overline{\Lambda_{c}^{k}(\Omega)},
\]
where 
\[
\Lambda_{c}^{k}(\Omega)\coloneqq\{T\in\Lambda^{k}(\Omega)\,;\,T\text{ compactly supported}\}.
\]
\end{rem}

We first inspect three different differentiation operators on $L_{2}^{k}(\Omega)$
and their relations, namely the Lie-derivative, the exterior derivative
and the covariant derivative. First, we define the operators to be
considered.
\begin{defn*}
Let $X\in T_{1}^{0}(\Omega)$ a vector field, $k\in\N$.

\begin{enumerate}[(a)]

\item We define the operator $\mathcal{L}_{X,c}$ by 
\[
\mathcal{L}_{X,c}:\Lambda_{c}^{k}(\Omega)\subseteq L_{2}^{k}(\Omega)\to L_{2}^{k}(\Omega),\quad T\mapsto\mathcal{L}_{X}T,
\]
where $\mathcal{L}_{X}$ denotes the Lie-derivative on $\Lambda^{k}(\Omega)$
in direction $X$. 

\item We define the operator $\nabla_{X,c}$ by 
\[
\nabla_{X,c}:\Lambda_{c}^{k}(\Omega)\subseteq L_{2}^{k}(\Omega)\to L_{2}^{k}(\Omega),\quad T\mapsto\nabla_{X}T,
\]
where $\mathcal{\nabla}_{X}$ denotes the covariant derivative on
$\Lambda^{k}(\Omega)$ in direction $X$.

\item We define the operator $d_{c}$ by 
\[
d_{c}:\Lambda_{c}^{k}(\Omega)\subseteq L_{2}^{k}(\Omega)\to L_{2}^{k+1}(\Omega),\;T\mapsto dT,
\]
where $d$ denotes the exterior derivative on $\Lambda^{k}(\Omega)$. 

\end{enumerate}
\end{defn*}
\begin{prop}
\label{prop:adjoint_derivative}Let $X\in T_{1}^{0}(\Omega).$ Then
\begin{align*}
-\ast\mathcal{L}_{X,c}\ast & \subseteq\mathcal{L}_{X,c}^{\ast},\\
-\nabla_{X,c} & \subseteq\nabla_{X,c}^{\ast},
\end{align*}
where $\ast:L_{2}^{k}(\Omega)\to L_{2}^{n-k}(\Omega)$ denotes the
Hodge-star operator. Moreover
\[
\left(-1\right)^{k+1}\ast d_{c}\ast\subseteq d_{c}^{\ast}
\]
as an operator from $L_{2}^{k+1}(\Omega)\to L_{2}^{k}(\Omega).$
\end{prop}

\begin{proof}
Let $S,T\in\Lambda_{c}^{k}(\Omega).$ Then we have, using the Hodge
star operator $\ast:L_{2}^{k}(\Omega)\to L_{2}^{n-k}(\Omega)$,
\begin{align*}
\langle\mathcal{L}_{X,c}S,T\rangle_{L_{2}^{k}(\Omega)} & =\int_{\Omega}\left(\mathcal{L}_{X}S\right)\wedge(\ast T)\d V\\
 & =\int_{\Omega}\mathcal{L}_{X}\left(S\wedge(\ast T)\right)-S\wedge\left(\mathcal{L}_{X}\ast T\right)\d V\\
 & =\langle S,-\ast\mathcal{L}_{X,c}\ast T\rangle_{L_{2}^{k}(\Omega)},
\end{align*}
which proves 
\[
-\ast\mathcal{L}_{X,c}\ast\subseteq\mathcal{L}_{X,c}^{\ast}.
\]
Here we have used Cartan's formula and Stokes' Theorem to compute
\begin{align*}
\int_{\Omega}\mathcal{L}_{X}\left(S\wedge(\ast T)\right)\d V & =\int_{\Omega}(d\iota_{X}+\iota_{X}d)(S\wedge(\ast T))\d V\\
 & =\int_{\Omega}d\iota_{X}(S\wedge(\ast T))\d V\\
 & =0,
\end{align*}
since $\iota_{X}(S\wedge(\ast T))\in\Lambda_{c}^{n-1}(\Omega)$, where
$\iota_{X}$ denotes the interior derivative in direction $X$. Similarly,
we compute 
\begin{align*}
\langle\nabla_{X,c}S,T\rangle_{L_{2}^{k}(\Omega)} & =\int_{\Omega}\nabla_{X,c}\langle S,T\rangle-\langle S,\nabla_{X,c}T\rangle\d V\\
 & =\int_{\Omega}\mathcal{L}_{X,c}(S\wedge(\ast T))\d V-\langle S,\nabla_{X,c}T\rangle_{L_{2}^{k}(\Omega)}\\
 & =\langle S,-\nabla_{X,c}T\rangle_{L_{2}^{k}(\Omega)},
\end{align*}
i.e. 
\[
-\nabla_{X,c}\subseteq\nabla_{X,c}^{\ast}.
\]
Finally, we compute for $S\in\Lambda_{c}^{k}(\Omega)$ and $T\in\Lambda_{c}^{k+1}(\Omega)$
\begin{align*}
\langle d_{c}S,T\rangle_{L_{2}^{k+1}(\Omega)} & =\int_{\Omega}(dS)\wedge(\ast T)\d V\\
 & =\int_{\Omega}d(S\wedge(\ast T))\d V-(-1)^{k}\int_{\Omega}S\wedge(d(\ast T))\d V\\
 & =\langle S,(-1)^{k+1}\ast d\ast T\rangle_{L_{2}^{k}(\Omega)}.
\end{align*}
\end{proof}
The latter proposition shows that each of the operators $\mathcal{L}_{X,c},\nabla_{X,c}$
and $d_{c}$ is closable and hence, we may define the following operators.
\begin{defn*}
Let $X\in T_{1}^{0}(M).$ We set 
\begin{align*}
\interior{\mathcal{L}}_{X} & \coloneqq\overline{\mathcal{L}_{X,c}},\quad\mathcal{L}_{X}\coloneqq-\ast\mathcal{L}_{X,c}^{\ast}\ast,\\
\interior{\nabla}_{X} & \coloneqq\overline{\nabla_{X,c}},\quad\nabla_{X}\coloneqq-\nabla_{X}^{\ast}
\end{align*}
and 
\[
\interior{d}\coloneqq\overline{d_{c}},\quad d\coloneqq(-1)^{n-k}\ast d_{c}^{\ast}\ast
\]
as operators from $L_{2}^{k}(\Omega)$ to $L_{2}^{k+1}(\Omega)$ for
$k\in\N_{<n}.$ 
\end{defn*}
\begin{rem}
By Proposition \ref{prop:adjoint_derivative} we have 
\[
\interior{\mathcal{L}}_{X}\subseteq\mathcal{L}_{X},\:\interior{\nabla}_{X}\subseteq\nabla_{X},\;\interior{d}\subseteq d.
\]
\end{rem}

\begin{lem}
\label{lem:Lie-covariant} Let $X\in T_{1}^{0}(M)$ and assume that
\[
C\coloneqq\sup\left\{ \|\nabla_{Y}X\|_{\infty}\,;\,Y\in T_{1,c}^{0}(\Omega),\:\|Y\|_{\infty}\leq1\right\} <\infty.
\]
Then $\mathcal{L}_{X}-\nabla_{X}$ is continuous. 
\end{lem}

\begin{proof}
By density, it suffices to prove that $\mathcal{L}_{X,c}-\nabla_{X,c}$
is continuous on $\Lambda_{c}^{k}(\Omega)$. Moreover, by induction
it suffices to show the assertion for $k=0$ and $k=1$. Since $\mathcal{L}_{X,c}$
and $\nabla_{X,c}$ agree on $\Lambda_{c}^{0}(\Omega)=C_{c}^{\infty}(\Omega),$
there is nothing to show for the case $k=0.$ So, let $\alpha\in\Lambda_{c}^{1}(\Omega).$
We then have for all $Y\in T_{1}^{0}(\Omega)$ with compact support
\begin{align*}
\left((\mathcal{L}_{X,c}\alpha)-(\nabla_{X,c}\alpha)\right)(Y) & =\mathcal{L}_{X}(\alpha(Y))-\alpha(\mathcal{L}_{X}Y)-\nabla_{X}(\alpha(Y))+\alpha(\nabla_{X}Y)\\
 & =\alpha(\nabla_{X}Y-\mathcal{L}_{X}Y)\\
 & =\alpha(\nabla_{Y}X),
\end{align*}
where we have used that $\nabla$ is torsion-free. Hence, 
\begin{align*}
\left\Vert (\mathcal{L}_{X,c}-\nabla_{X,c})(\alpha)\right\Vert _{L_{2}^{1}(\Omega)}^{2} & =\int_{\Omega}\|(\mathcal{L}_{X,c}-\nabla_{X,c})(\alpha)(p)\|_{\Lambda^{1}(T_{p}M)}^{2}\d V(p)\\
 & =\int_{\Omega}\sup_{Z\in T_{p}M,\|Z\|\leq1}\left|\langle(\mathcal{L}_{X,c}-\nabla_{X,c})(\alpha)(p),Z\rangle\right|^{2}\d V(p)\\
 & =\sup_{Y\in T_{1,c}^{0}(\Omega),\|Y\|_{\infty}\leq1}\int_{\Omega}|\alpha(\nabla_{Y}X)(p)|^{2}\d V(p)\\
 & \leq\sup_{Y\in T_{1,c}^{0}(\Omega),\|Y\|_{\infty}\leq1}\int_{\Omega}\|\alpha(p)\|^{2}\|\nabla_{Y}X\|_{\infty}^{2}\d V(p)\\
 & =C^{2}\|\alpha\|_{L_{2}^{1}(\Omega)}^{2},
\end{align*}
which shows the claim. 
\end{proof}
~
\begin{lem}
\label{lem:exterior_Lie_commute}Let $X\in T_{1}^{0}(M).$ It is
\[
d\mathcal{L}_{X}=\mathcal{L}_{X}d
\]
on $\Lambda^{k}(\Omega)$.
\end{lem}

\begin{proof}
We apply Cartan's magic formula stating that 
\[
\mathcal{L}_{X}=d\iota_{X}+\iota_{X}d
\]
on $\Lambda^{k}(\Omega).$ Thus, 
\begin{align*}
d\mathcal{L}_{X} & =dd\iota_{X}+d\iota_{X}d\\
 & =d\iota_{X}d\\
 & =(d\iota_{X}+\iota_{X}d)d\\
 & =\mathcal{L}_{X}d
\end{align*}
on $\Lambda^{k}(\Omega)$.
\end{proof}
\begin{prop}
\label{prop:Lie_exterior}Let $\mathcal{L}_{X}$ be quasi-skew-selfadjoint.
Moreover, let $\eta\in\R$ with $|\eta|$ small enough and assume
that 
\[
\left(1+\eta\mathcal{L}_{X}\right)\left[\Lambda^{k}(\Omega)\cap\dom(\mathcal{L}_{X})\cap\dom(d)\right]\cap\dom(\interior{d})\text{ is dense in }\dom(\interior d).
\]
Then 
\[
\left(1+\eta\mathcal{L}_{X}\right)^{-1}\interior d\subseteq\interior d\left(1+\eta\mathcal{L}_{X}\right)^{-1}.
\]
\end{prop}

\begin{proof}
Since $\mathcal{L}_{X}$ is quasi-skew-selfadjoint and hence quasi-m-accretive,
$(1+\varepsilon\mathcal{L}_{X})^{-1}$ defines a bounded operator
on $L_{2}^{k}(\Omega)$ for $\varepsilon>0$ small enough. \\
Let now $\alpha\in\Lambda^{k}(\Omega)\cap\dom(\mathcal{L}_{X})\cap\dom(d)$
such that $(1+\eta\mathcal{L}_{X})\alpha\in\dom(d)$. By Lemma \ref{lem:exterior_Lie_commute}
we have that 
\begin{align*}
 & d(1+\eta\mathcal{L}_{X})^{-1}(1+\eta\mathcal{L}_{X})\alpha-(1+\eta\mathcal{L}_{X})^{-1}d(1+\eta\mathcal{L}_{X})\alpha\\
= & d\alpha-d\alpha=0,
\end{align*}
i.e. 
\[
d(1+\eta\mathcal{L}_{X})^{-1}-(1+\eta\mathcal{L}_{X})^{-1}d=0
\]
on $(1+\eta\mathcal{L}_{X})[\Lambda^{k}(\Omega)\cap\dom(\mathcal{L}_{X})\cap\dom(d)]\cap\dom(d).$
Let now $\alpha\in\dom(\interior{d})$ and $(\alpha_{n})_{n}$ in
$(1+\eta\mathcal{L}_{X})[\Lambda^{k}(\Omega)\cap\dom(\mathcal{L}_{X})\cap\dom(d)]\cap\dom(\interior{d})$
with $\alpha_{n}\to\alpha$ in $\dom(\interior{d}).$ Then 
\[
\interior{d}(1+\eta\mathcal{L}_{X})^{-1}\alpha_{n}=(1+\eta\mathcal{L}_{X})^{-1}\interior{d}\alpha_{n}\to(1+\eta\mathcal{L}_{X})^{-1}\interior{d}\alpha
\]
as $n\to\infty.$ Since $(1+\eta\mathcal{L}_{X})^{-1}\alpha_{n}\to(1+\eta\mathcal{L}_{X})^{-1}\alpha$
as $n\to\infty,$ we infer that $(1+\eta\mathcal{L}_{X})^{-1}\alpha\in\dom(\interior{d})$
and 
\[
\interior{d}(1+\eta\mathcal{L}_{X})^{-1}\alpha=(1+\eta\mathcal{L}_{X})^{-1}\interior{d}\alpha,
\]
which shows the asserted operator inclusion.
\end{proof}

\subsection{The Equations on smooth Riemannian manifolds\label{subsec:The-Equations-on}}

We now come to the equation we want to study. Let $\Omega\subseteq M$
be open for a smooth Riemannian manifold $M$ of odd dimension $n$.
Moreover, let $1\leq k<n$ and set $H\coloneqq\Lambda^{k}(\Omega)\times\Lambda^{k+1}(\Omega).$
We assume that $M_{0},M_{1}\in L(H)$ such that $M_{0}$ is selfadjoint
and $M_{0}\geq c>0$. Moreover, let $X_{0}\in T_{1}^{0}(M)$ and $\alpha\in L_{\infty}(\Omega;\R)$.
We consider the equation 
\begin{equation}
\left(\partial_{0}M_{0}+\alpha\left(\begin{array}{cc}
\nabla_{X_{0}} & 0\\
0 & \nabla_{X_{0}}
\end{array}\right)M_{0}+M_{1}+\left(\begin{array}{cc}
0 & -\interior{d}^{*}\\
\interior{d} & 0
\end{array}\right)\right)U=F\label{eq:prob}
\end{equation}
for given $F\in L_{2,\rho}(\R;H)$ for some $\rho\in\R$ big enough,
where we identify $\alpha$ with its induced multiplication operator
in $H$. We impose the following conditions on the vector field $X_{0}$
and the operators $\alpha$ and $M_{0}.$

\begin{hyp}\label{hyp:standard} We assume that 
\[
\sup\left\{ \|\nabla_{Y}X_{0}\|_{\infty}\,;\,Y\in T_{1,c}^{0}(\Omega),\:\|Y\|_{\infty}\leq1\right\} <\infty
\]
as well as $\mathcal{L}_{x_{0}}$ is quasi-skew-selfadjoint. Moreover,
we assume that $M_{0}[\dom(\mathcal{L}_{X_{0}})\times\dom(\mathcal{L}_{X_{0}})]\subseteq\dom(\mathcal{L}_{X_{0}})\times\dom(\mathcal{L}_{X_{0}})$,
$\alpha[\dom(\mathcal{L}_{X_{0}})]\subseteq\dom(\mathcal{L}_{X_{0}})$
as well as 
\[
\left[\left(\begin{array}{cc}
\mathcal{L}_{X_{0}} & 0\\
0 & \mathcal{L}_{X_{0}}
\end{array}\right),M_{0}\right]\text{ and }\left[\mathcal{L}_{X_{0}},\alpha\right]
\]
are continuous. Finally, we assume that 
\[
(1\pm\epsilon\mathcal{L}_{X})[\Lambda^{k}(\Omega)\cap\dom(\mathcal{L}_{X})\cap\dom(d)]\cap\dom(\interior{d})\text{ is dense in }\dom(\interior d)
\]
for all $\varepsilon>0$ small enough and that $M_{0}\alpha=\alpha M_{0}$.

\end{hyp}

Under this hypotheses we observe the following operator relations.
\begin{lem}
\label{lem:first-reduction}The operators $\left(\begin{array}{cc}
\nabla_{X_{0}} & 0\\
0 & \nabla_{X_{0}}
\end{array}\right)-\left(\begin{array}{cc}
\mathcal{L}_{X_{0}} & 0\\
0 & -\mathcal{L}_{X_{0}}^{*}
\end{array}\right)$, $\left[\left(\begin{array}{cc}
\mathcal{L}_{X_{0}} & 0\\
0 & -\mathcal{L}_{X_{0}}^{*}
\end{array}\right),M_{0}\right]$ and $[\alpha,\mathcal{L}_{X_{0}}^{\ast}]$ are continuous and densely
defined.
\end{lem}

\begin{proof}
We obtain 
\[
\left(\begin{array}{cc}
\nabla_{X_{0}} & 0\\
0 & \nabla_{X_{0}}
\end{array}\right)-\left(\begin{array}{cc}
\mathcal{L}_{X_{0}} & 0\\
0 & -\mathcal{L}_{X_{0}}^{*}
\end{array}\right)=\left(\begin{array}{cc}
\nabla_{X_{0}}-\mathcal{L}_{X_{0}} & 0\\
0 & \nabla_{X_{0}}-\mathcal{L}_{X_{0}}
\end{array}\right)+\left(\begin{array}{cc}
0 & 0\\
0 & \mathcal{L}_{X_{0}}+\mathcal{L}_{X_{0}}^{*}
\end{array}\right),
\]
which is continuous by Lemma \ref{lem:Lie-covariant} and the quasi-skew-selfadjointness
of $\mathcal{L}_{X_{0}}$. For the second operator we compute 
\[
\left[\left(\begin{array}{cc}
\mathcal{L}_{X_{0}} & 0\\
0 & -\mathcal{L}_{X_{0}}^{*}
\end{array}\right),M_{0}\right]=\left[\left(\begin{array}{cc}
\mathcal{L}_{X_{0}} & 0\\
0 & \mathcal{L}_{X_{0}}
\end{array}\right),M_{0}\right]-\left[\left(\begin{array}{cc}
0 & 0\\
0 & \mathcal{L}_{X_{0}}^{\ast}+\mathcal{L}_{X_{0}}
\end{array}\right),M_{0}\right],
\]
and both operators on the right-hand side are continuous, which yields
the assertion. The continuity of $[\alpha,\mathcal{L}_{X_{0}}^{\ast}]$
follows by arguing in the same way.
\end{proof}
In order to study (\ref{eq:prob}), we rewrite the second operator
on the left-hand side in the following way

\begin{align*}
 & \alpha\left(\begin{array}{cc}
\nabla_{X_{0}} & 0\\
0 & \nabla_{X_{0}}
\end{array}\right)M_{0}\\
 & =\alpha\left(\begin{array}{cc}
\mathcal{L}_{X_{0}} & 0\\
0 & -\mathcal{L}_{X_{0}}^{*}
\end{array}\right)M_{0}+\alpha\left(\overline{\left(\begin{array}{cc}
\nabla_{X_{0}} & 0\\
0 & \nabla_{X_{0}}
\end{array}\right)-\left(\begin{array}{cc}
\mathcal{L}_{X_{0}} & 0\\
0 & -\mathcal{L}_{X_{0}}^{*}
\end{array}\right)}\right)M_{0}\\
 & =\alpha M_{0}\left(\begin{array}{cc}
\mathcal{L}_{X_{0}} & 0\\
0 & -\mathcal{L}_{X_{0}}^{*}
\end{array}\right)+\alpha\left(\overline{\left(\begin{array}{cc}
\nabla_{X_{0}} & 0\\
0 & \nabla_{X_{0}}
\end{array}\right)-\left(\begin{array}{cc}
\mathcal{L}_{X_{0}} & 0\\
0 & -\mathcal{L}_{X_{0}}^{*}
\end{array}\right)}\right)M_{0}+\alpha\overline{\left[\left(\begin{array}{cc}
\mathcal{L}_{X_{0}} & 0\\
0 & -\mathcal{L}_{X_{0}}^{*}
\end{array}\right),M_{0}\right]},
\end{align*}
where we have used Lemma \ref{lem:first-reduction}. Thus, ignoring
the bounded operators, we may restrict ourselves to the study of the
operator 
\[
\partial_{0}M_{0}+\alpha M_{0}\left(\begin{array}{cc}
\mathcal{L}_{X_{0}} & 0\\
0 & -\mathcal{L}_{X_{0}}^{*}
\end{array}\right)+\left(\begin{array}{cc}
0 & -\interior{d}^{\ast}\\
\interior{d} & 0
\end{array}\right).
\]
The main idea is now to decompose the second operator in the above
sum in its symmetric and skew-symmetric part, which are studied in
the next proposition.
\begin{prop}
\label{prop:decomposition}The operator
\[
C\coloneqq\sym\alpha M_{0}\left(\begin{array}{cc}
\mathcal{L}_{X_{0}} & 0\\
0 & -\mathcal{L}_{X_{0}}^{*}
\end{array}\right)
\]
is bounded and selfadjoint. Moreover 
\[
D\coloneqq\skew\alpha M_{0}\left(\begin{array}{cc}
\mathcal{L}_{X_{0}} & 0\\
0 & -\mathcal{L}_{X_{0}}^{*}
\end{array}\right)
\]
is skew-selfadjoint.
\end{prop}

\begin{proof}
We note that $\left(\begin{array}{cc}
\mathcal{L}_{X_{0}} & 0\\
0 & -\mathcal{L}_{X_{0}}^{\ast}
\end{array}\right)$ is quasi-skew-selfadjoint in the sense of Subsection \ref{subsec:Commutators-for-Quasi-Skew-Selfa}.
Indeed, we have that 
\[
\dom\left(\begin{array}{cc}
\mathcal{L}_{X_{0}} & 0\\
0 & -\mathcal{L}_{X_{0}}^{\ast}
\end{array}\right)=\dom(\mathcal{L}_{X_{0}})\times\dom(\mathcal{L}_{X_{0}}^{\ast})=\dom(\mathcal{L}_{X_{0}})\times\dom(\mathcal{L}_{X_{0}})=\dom\left(\begin{array}{cc}
\mathcal{L}_{X_{0}}^{\ast} & 0\\
0 & -\mathcal{L}_{X_{0}}
\end{array}\right)
\]
since $\mathcal{L}_{X_{0}}$ is quasi-skew-selfadjoint and 
\[
\left(\begin{array}{cc}
\mathcal{L}_{X_{0}} & 0\\
0 & -\mathcal{L}_{X_{0}}^{\ast}
\end{array}\right)+\left(\begin{array}{cc}
\mathcal{L}_{X_{0}}^{\ast} & 0\\
0 & -\mathcal{L}_{X_{0}}
\end{array}\right)=\left(\begin{array}{cc}
\mathcal{L}_{X_{0}}+\mathcal{L}_{X_{0}}^{\ast} & 0\\
0 & -(\mathcal{L}_{X_{0}}^{\ast}+\mathcal{L}_{X_{0}})
\end{array}\right)
\]
is bounded, again by the quasi-skew-selfadjointness of $\mathcal{L}_{X_{0}}.$
\\
Next we note that $\alpha M_{0}$ is selfadjoint and that 
\[
\alpha M_{0}[\dom(\mathcal{L}_{X_{0}})\times\dom(\mathcal{L}_{X_{0}})]\subseteq\dom(\mathcal{L}_{X_{0}})\times\dom(\mathcal{L}_{X_{0}})
\]
by assumption. Moreover, 
\[
\left[\alpha M_{0},\left(\begin{array}{cc}
\mathcal{L}_{X_{0}} & 0\\
0 & -\mathcal{L}_{X_{0}}^{\ast}
\end{array}\right)\right]=\alpha\left[M_{0},\left(\begin{array}{cc}
\mathcal{L}_{X_{0}} & 0\\
0 & -\mathcal{L}_{X_{0}}^{\ast}
\end{array}\right)\right]+\left[\alpha,\left(\begin{array}{cc}
\mathcal{L}_{X_{0}} & 0\\
0 & -\mathcal{L}_{X_{0}}^{\ast}
\end{array}\right)\right]M_{0}
\]
is continuous by Lemma \ref{lem:first-reduction} and hence, $D$
is skew-selfadjoint by Theorem \ref{thm:quasi-skewselfadjoint}. Moreover,
$C$ is densely defined, since $\dom(\mathcal{L}_{X_{0}})\times\dom(\mathcal{L}_{X_{0}})\subseteq\dom(C)$
and 
\begin{align*}
 & \alpha M_{0}\left(\begin{array}{cc}
\mathcal{L}_{X_{0}} & 0\\
0 & -\mathcal{L}_{X_{0}}^{\ast}
\end{array}\right)+\left(\begin{array}{cc}
\mathcal{L}_{X_{0}}^{\ast} & 0\\
0 & -\mathcal{L}_{X_{0}}
\end{array}\right)M_{0}\alpha\\
= & \left[\alpha M_{0},\left(\begin{array}{cc}
\mathcal{L}_{X_{0}} & 0\\
0 & -\mathcal{L}_{X_{0}}^{\ast}
\end{array}\right)\right]+\left(\begin{array}{cc}
\mathcal{L}_{X_{0}} & 0\\
0 & -\mathcal{L}_{X_{0}}^{\ast}
\end{array}\right)\alpha M_{0}+\left(\begin{array}{cc}
\mathcal{L}_{X_{0}}^{\ast} & 0\\
0 & -\mathcal{L}_{X_{0}}
\end{array}\right)M_{0}\alpha\\
= & \left[\alpha M_{0},\left(\begin{array}{cc}
\mathcal{L}_{X_{0}} & 0\\
0 & -\mathcal{L}_{X_{0}}^{\ast}
\end{array}\right)\right]+\left(\begin{array}{cc}
\mathcal{L}_{X_{0}}+\mathcal{L}_{X_{0}}^{\ast} & 0\\
0 & -(\mathcal{L}_{X_{0}}+\mathcal{L}_{X_{0}}^{\ast})
\end{array}\right)\alpha M_{0}
\end{align*}
shows that $C$ is continuous and hence, selfadjoint.
\end{proof}
It is now our goal to prove that the operator 
\[
D+\left(\begin{array}{cc}
0 & -\interior{d}^{\ast}\\
\interior{d} & 0
\end{array}\right)
\]
with $D$ given as in Proposition \ref{prop:decomposition}, is essentially
skew-selfadjoint; i.e., its closure is skew-selfadjoint. For doing
so, we want to apply Theorem \ref{thm:sum} with 
\begin{equation}
L_{\varepsilon}\coloneqq\left(\begin{array}{cc}
\left(1-\epsilon\mathcal{L}_{X_{0}}^{*}\right)^{-1} & 0\\
0 & \left(1+\epsilon\mathcal{L}_{X_{0}}\right)^{-1}
\end{array}\right),\quad R_{\varepsilon}\coloneqq\left(\begin{array}{cc}
\left(1+\epsilon\mathcal{L}_{X_{0}}\right)^{-1} & 0\\
0 & \left(1-\epsilon\mathcal{L}_{X_{0}}^{*}\right)^{-1}
\end{array}\right)\label{eq:L and R}
\end{equation}
for $\epsilon>0$ small enough. Note that these operators are well-defined
and bounded, since by Proposition \ref{prop:Lie_exterior} the operator
$\mathcal{L}_{X_{0}}$ is quasi-skew-selfadjoint and hence, $\pm\mathcal{L}_{X_{0}}$
and $\pm\mathcal{L}_{X_{0}}^{\ast}$ are quasi-m-accretive. Note further
that this yields 
\[
L_{\varepsilon}^{\ast}=\left(\begin{array}{cc}
\left(1-\epsilon\mathcal{L}_{X_{0}}\right)^{-1} & 0\\
0 & \left(1+\epsilon\mathcal{L}_{X_{0}}^{\ast}\right)^{-1}
\end{array}\right)\to1_{H}
\]
strongly as $\varepsilon\to0$ by Lemma \ref{lem:resolvent_m_accretive}
and similarly $R_{\varepsilon}^{\ast}\to1_{H}$ strongly as $\varepsilon\to0.$ 
\begin{lem}
\label{lem:pre-sum}Let $\varepsilon>0$ small enough and $D$ as
in Proposition \ref{prop:decomposition}. Then the following statements
hold:

\begin{enumerate}[(a)]

\item We have 
\[
\ran(R_{\varepsilon})\subseteq\dom\left(\left(\begin{array}{cc}
\mathcal{L}_{X_{0}}^{\ast} & 0\\
0 & -\mathcal{L}_{X_{0}}
\end{array}\right)M_{0}\alpha\right)\cap\dom\left(\alpha M_{0}\left(\begin{array}{cc}
\mathcal{L}_{X_{0}} & 0\\
0 & -\mathcal{L}_{X_{0}}^{\ast}
\end{array}\right)\right)\subseteq\dom(D)
\]
and 
\[
R_{\varepsilon}\left[\dom(\interior{d})\times\dom(\interior{d}^{\ast})\right]\subseteq\dom(\interior{d})\times\dom(\interior{d}^{\ast}).
\]

\item We have
\[
\left(D+\left(\begin{array}{cc}
0 & -\interior{d}^{\ast}\\
\interior{d} & 0
\end{array}\right)\right)^{\ast}L_{\varepsilon}^{\ast}=-\left(D+\left(\begin{array}{cc}
0 & -\interior{d}^{\ast}\\
\interior{d} & 0
\end{array}\right)\right)L_{\varepsilon}^{\ast}.
\]

\end{enumerate}
\end{lem}

\begin{proof}
\begin{enumerate}[(a)]

\item Note that $\ran(R_{\varepsilon})\subseteq\dom(\mathcal{L}_{X_{0}})\times\dom(\mathcal{L}_{X_{0}}^{\ast})=\dom(\mathcal{L}_{X_{0}})\times\dom(\mathcal{L}_{X_{0}}).$
Moreover, we have that 
\[
\dom\left(\left(\begin{array}{cc}
\mathcal{L}_{X_{0}}^{\ast} & 0\\
0 & -\mathcal{L}_{X_{0}}
\end{array}\right)M_{0}\alpha\right)\supseteq\dom(\mathcal{L}_{X_{0}})\times\dom(\mathcal{L}_{X_{0}})
\]
and hence, 
\begin{align}
\ran(R_{\varepsilon}) & \subseteq\dom(\mathcal{L}_{X_{0}})\times\dom(\mathcal{L}_{X_{0}})\nonumber \\
 & =\dom\left(\left(\begin{array}{cc}
\mathcal{L}_{X_{0}}^{\ast} & 0\\
0 & -\mathcal{L}_{X_{0}}
\end{array}\right)M_{0}\alpha\right)\cap\dom\left(\alpha M_{0}\left(\begin{array}{cc}
\mathcal{L}_{X_{0}} & 0\\
0 & -\mathcal{L}_{X_{0}}^{\ast}
\end{array}\right)\right)\nonumber \\
 & \subseteq\dom(D).\label{eq:dom(L_x) in dom(D)}
\end{align}
Moreover, we have $(1\pm\varepsilon\mathcal{L}_{X})^{-1}\interior{d}\subseteq\interior{d}(1\pm\varepsilon\mathcal{L}_{X})^{-1}$
by Proposition \ref{prop:Lie_exterior}, which implies $(1\pm\varepsilon\mathcal{L}_{X}^{\ast})^{-1}\interior{d}^{\ast}\subseteq\interior{d}^{\ast}(1\pm\varepsilon\mathcal{L}_{X}^{\ast})^{-1}$.
Hence 
\[
R_{\varepsilon}\left[(\dom(\interior{d})\times\dom(\interior{d}^{\ast}))\right]\subseteq\dom(\interior{d})\times\dom(\interior{d}^{\ast}),
\]
which completes the proof for statement (a).

\item It suffices to prove 
\[
\dom\left(\left(D+\left(\begin{array}{cc}
0 & -\interior{d}^{\ast}\\
\interior{d} & 0
\end{array}\right)\right)^{\ast}L_{\varepsilon}^{\ast}\right)\subseteq\dom\left(\left(D+\left(\begin{array}{cc}
0 & -\interior{d}^{\ast}\\
\interior{d} & 0
\end{array}\right)\right)L_{\varepsilon}^{\ast}\right).
\]
So let $(x,y)\in\dom\left(\left(D+\left(\begin{array}{cc}
0 & -\interior{d}^{\ast}\\
\interior{d} & 0
\end{array}\right)\right)^{\ast}L_{\varepsilon}^{\ast}\right);$ i.e, 
\[
L_{\varepsilon}^{\ast}\left(x,y\right)=\left(\begin{array}{c}
\left(1-\epsilon\mathcal{L}_{X_{0}}\right)^{-1}x\\
\left(1+\epsilon\mathcal{L}_{X_{0}}^{\ast}\right)^{-1}y
\end{array}\right)\in\dom\left(\left(D+\left(\begin{array}{cc}
0 & -\interior{d}^{\ast}\\
\interior{d} & 0
\end{array}\right)\right)^{\ast}\right).
\]
Since $L_{\varepsilon}^{\ast}\left(x,y\right)\in\dom(\mathcal{L}_{X_{0}})\times\dom(\mathcal{L}_{X_{0}})\subseteq\dom(D)=\dom(D^{\ast})$
by (\ref{eq:dom(L_x) in dom(D)}), the assertion would follow if $\dom(D)$
is a core for $\left(\begin{array}{cc}
0 & -\interior{d}^{\ast}\\
\interior{d} & 0
\end{array}\right)$. Since $(1+\varepsilon\mathcal{L}_{X})^{-1}\interior{d}\subseteq\interior{d}(1+\varepsilon\mathcal{L}_{X})^{-1}$
and $(1+\varepsilon\mathcal{L}_{X})^{-1}\to1$ strongly as $\varepsilon\to0$,
it follows that $\dom(\mathcal{L}_{X})$ is a core for $\interior{d}.$
In the same way, it follows that $\dom(\mathcal{L}_{X}^{\ast})=\dom(\mathcal{L}_{X})$
is a core for $\interior{d}^{\ast}$ and hence, since $\dom(\mathcal{L}_{X})\times\dom(\mathcal{L}_{X})\subseteq\dom(D)$
by (\ref{eq:dom(L_x) in dom(D)}) the claim follows. This proves statement
(b).\qedhere

\end{enumerate}
\end{proof}
\begin{lem}
\label{lem:transmutator-bdd}For $\varepsilon_{0}>0$ small enough
there exists $K\geq0$ such that 
\[
\|[R_{\varepsilon}^{\ast},\alpha M_{0},L_{\varepsilon}^{\ast}]\|\leq K2\varepsilon\|C\|\quad(0<\varepsilon\leq\varepsilon_{0}),
\]
where $C$ is the operator given in Proposition \ref{prop:decomposition}.
\end{lem}

\begin{proof}
We first observe that 
\begin{align*}
\epsilon\left[\left(\begin{array}{cc}
\mathcal{L}_{X_{0}}^{*} & 0\\
0 & -\mathcal{L}_{X_{0}}
\end{array}\right),\alpha M_{0},\left(\begin{array}{cc}
-\mathcal{L}_{X_{0}} & 0\\
0 & \mathcal{L}_{X_{0}}^{*}
\end{array}\right)\right] & =\left(\begin{array}{cc}
\left(1+\epsilon\mathcal{L}_{X_{0}}^{*}\right) & 0\\
0 & \left(1-\epsilon\mathcal{L}_{X_{0}}\right)
\end{array}\right)\alpha M_{0}-\\
 & \quad-\alpha M_{0}\left(\begin{array}{cc}
\left(1-\epsilon\mathcal{L}_{X_{0}}\right) & 0\\
0 & \left(1+\epsilon\mathcal{L}_{X_{0}}^{*}\right)
\end{array}\right)
\end{align*}
and thus, 
\[
\]
\begin{align*}
[R_{\varepsilon}^{\ast},\alpha M_{0},L_{\varepsilon}^{\ast}] & =-\epsilon R_{\varepsilon}^{\ast}\left[\left(\begin{array}{cc}
\mathcal{L}_{X_{0}}^{*} & 0\\
0 & -\mathcal{L}_{X_{0}}
\end{array}\right),\alpha M_{0},\left(\begin{array}{cc}
-\mathcal{L}_{X_{0}} & 0\\
0 & \mathcal{L}_{X_{0}}^{*}
\end{array}\right)\right]L_{\varepsilon}^{\ast}\\
 & =-\epsilon R_{\varepsilon}^{\ast}\left(\left(\begin{array}{cc}
\mathcal{L}_{X_{0}}^{*} & 0\\
0 & -\mathcal{L}_{X_{0}}
\end{array}\right)\alpha M_{0}-\alpha M_{0}\left(\begin{array}{cc}
-\mathcal{L}_{X_{0}} & 0\\
0 & \mathcal{L}_{X_{0}}^{*}
\end{array}\right)\right)L_{\varepsilon}^{\ast}\\
 & =-\epsilon R_{\epsilon}^{\ast}\left(\alpha M_{0}\left(\begin{array}{cc}
\mathcal{L}_{X_{0}} & 0\\
0 & -\mathcal{L}_{X_{0}}^{*}
\end{array}\right)+\left(\begin{array}{cc}
\mathcal{L}_{X_{0}}^{*} & 0\\
0 & -\mathcal{L}_{X_{0}}
\end{array}\right)M_{0}\alpha\right)\\
 & =-2\varepsilon R_{\varepsilon}^{\ast}CL_{\varepsilon}^{\ast}.
\end{align*}
The assertion follows with $K\coloneqq\sup_{0<\varepsilon\leq\varepsilon_{0}}\|R_{\varepsilon}^{\ast}\|\|L_{\varepsilon}^{\ast}\|,$
which is finite, since $\mathcal{L}_{X}$ is quasi-skew-selfadjoint
and thus, $-\mathcal{L}_{X},\mathcal{L}_{X}^{\ast}$ are quasi-m-accretive.
\end{proof}
With these preparations at hand, we are able to prove the essentially
skew-selfadjointness of $D+\left(\begin{array}{cc}
0 & -\interior{d}^{\ast}\\
\interior{d} & 0
\end{array}\right)$.
\begin{prop}
\label{prop:sum-skew-selfadjoint}The operator $D+\left(\begin{array}{cc}
0 & -\interior{d}^{\ast}\\
\interior{d} & 0
\end{array}\right)$ is essentially skew-selfadjoint, where $D$ is given in Proposition
\ref{prop:decomposition}.
\end{prop}

\begin{proof}
We will apply Theorem \ref{thm:sum} with the operators $L_{\varepsilon},R_{\varepsilon}$
given in (\ref{eq:L and R}). Thanks to Lemma \ref{lem:pre-sum} we
only need to check that 
\[
\left[L_{\varepsilon},D+\left(\begin{array}{cc}
0 & -\interior{d}^{\ast}\\
\interior{d} & 0
\end{array}\right),R_{\varepsilon}\right]
\]
is continuous and that 
\[
\left[L_{\varepsilon},D+\left(\begin{array}{cc}
0 & -\interior{d}^{\ast}\\
\interior{d} & 0
\end{array}\right),R_{\varepsilon}\right]^{\ast}\to0
\]
strongly as $\epsilon\to0$. We compute 
\begin{align*}
\left[L_{\varepsilon},D+\left(\begin{array}{cc}
0 & -\interior{d}^{\ast}\\
\interior{d} & 0
\end{array}\right),R_{\varepsilon}\right] & =\left[L_{\varepsilon},D,R_{\varepsilon}\right]+\left[L_{\varepsilon},\left(\begin{array}{cc}
0 & -\interior{d}^{\ast}\\
\interior{d} & 0
\end{array}\right),R_{\varepsilon}\right]\\
 & =[L_{\varepsilon},D,R_{\varepsilon}],
\end{align*}
where we have used $(1\pm\varepsilon\mathcal{L}_{X_{0}})^{-1}\interior{d}\subseteq\interior{d}(1\pm\varepsilon\mathcal{L}_{X_{0}})^{-1}$
by Proposition \ref{prop:Lie_exterior}, which implies $(1\pm\varepsilon\mathcal{L}_{X_{0}}^{\ast})^{-1}\interior{d}^{\ast}\subseteq\interior{d}^{\ast}(1\pm\varepsilon\mathcal{L}_{X_{0}}^{\ast})^{-1}$
and thus, the second transmutator vanishes. Now we have 
\begin{align*}
 & \frac{1}{2}\left(\alpha M_{0}\left(\begin{array}{cc}
\mathcal{L}_{X_{0}} & 0\\
0 & -\mathcal{L}_{X_{0}}^{\ast}
\end{array}\right)-\left(\begin{array}{cc}
\mathcal{L}_{X_{0}}^{\ast} & 0\\
0 & -\mathcal{L}_{X_{0}}
\end{array}\right)M_{0}\alpha\right)\left(\begin{array}{cc}
(1\pm\varepsilon\mathcal{L}_{X_{0}})^{-1} & 0\\
0 & (1\mp\varepsilon\mathcal{L}_{X_{0}}^{\ast})^{-1}
\end{array}\right)\\
 & =\frac{1}{2}\left(\left[\alpha M_{0},\left(\begin{array}{cc}
\mathcal{L}_{X_{0}}^{\ast} & 0\\
0 & -\mathcal{L}_{X_{0}}
\end{array}\right)\right]+\alpha M_{0}\left(\begin{array}{cc}
\mathcal{L}_{X_{0}}-\mathcal{L}_{X_{0}}^{\ast} & 0\\
0 & \mathcal{L}_{X_{0}}-\mathcal{L}_{X_{0}}^{\ast}
\end{array}\right)\right)\left(\begin{array}{cc}
(1\pm\varepsilon\mathcal{L}_{X_{0}})^{-1} & 0\\
0 & (1\mp\varepsilon\mathcal{L}_{X_{0}}^{\ast})^{-1}
\end{array}\right)\\
 & =\frac{1}{2}\left[\alpha M_{0},\left(\begin{array}{cc}
\mathcal{L}_{X_{0}}^{\ast} & 0\\
0 & -\mathcal{L}_{X_{0}}
\end{array}\right)\right]\left(\begin{array}{cc}
(1\pm\varepsilon\mathcal{L}_{X_{0}})^{-1} & 0\\
0 & (1\mp\varepsilon\mathcal{L}_{X_{0}}^{\ast})^{-1}
\end{array}\right)+\\
 & \quad+\alpha M_{0}\left(\begin{array}{cc}
\left(\mathcal{L}_{X_{0}}-\sym(\mathcal{L}_{X_{0}})\right)(1\pm\varepsilon\mathcal{L}_{X_{0}})^{-1} & 0\\
0 & \left(-\mathcal{L}_{X_{0}}^{\ast}+\sym(\mathcal{L}_{X_{0}})\right)(1\mp\epsilon\mathcal{L}_{X_{0}}^{\ast})^{-1}
\end{array}\right),
\end{align*}
which is continuous. Note that in one case this operators equals $DR_{\varepsilon}$
and in the other case it equals $DL_{\varepsilon}^{\ast}.$ Since
$L_{\varepsilon}D\subseteq(D^{\ast}L_{\varepsilon}^{\ast})^{\ast}=\left(-DL_{\varepsilon}^{\ast}\right)^{\ast}$,
we infer that also $L_{\varepsilon}D$ is continuous and hence, so
is $[L_{\varepsilon},D,R_{\varepsilon}].$\\
Now we come to the second claim. We compute 
\begin{align*}
[L_{\varepsilon},D,R_{\varepsilon}]^{\ast} & =\left(L_{\varepsilon}D-DR_{\varepsilon}\right)^{\ast}\\
 & =(L_{\varepsilon}D)^{\ast}-(DR_{\varepsilon})^{\ast}\\
 & =D^{\ast}L_{\varepsilon}^{\ast}-(DR_{\varepsilon})^{\ast}\\
 & =-(DL_{\varepsilon}^{\ast}+(DR_{\varepsilon})^{\ast})
\end{align*}
and so, we have to show that $DL_{\varepsilon}^{\ast}+(DR_{\varepsilon})^{\ast}\to0$
strongly as $\varepsilon\to0.$ By the computation above we have that
\begin{align*}
DL_{\varepsilon}^{\ast} & =\frac{1}{2}\left[\alpha M_{0},\left(\begin{array}{cc}
\mathcal{L}_{X_{0}}^{\ast} & 0\\
0 & -\mathcal{L}_{X_{0}}
\end{array}\right)\right]L_{\varepsilon}^{\ast}+\alpha M_{0}\left(\begin{array}{cc}
\mathcal{L}_{X_{0}} & 0\\
0 & -\mathcal{L}_{X_{0}}^{\ast}
\end{array}\right)L_{\varepsilon}^{\ast}+\alpha M_{0}\left(\begin{array}{cc}
-\sym(\mathcal{L}_{X_{0}}) & 0\\
0 & \sym(\mathcal{L}_{X_{0}})
\end{array}\right)L_{\varepsilon}^{\ast}
\end{align*}
and 
\begin{align*}
(DR_{\varepsilon})^{\ast} & =R_{\varepsilon}^{\ast}\frac{1}{2}\left[\alpha M_{0},\left(\begin{array}{cc}
\mathcal{L}_{X_{0}}^{\ast} & 0\\
0 & -\mathcal{L}_{X_{0}}
\end{array}\right)\right]^{\ast}+R_{\varepsilon}^{\ast}\left(\begin{array}{cc}
\mathcal{L}_{X_{0}}^{\ast} & 0\\
0 & -\mathcal{L}_{X_{0}}
\end{array}\right)M_{0}\alpha+\\
 & \quad+R_{\varepsilon}^{\ast}\left(\begin{array}{cc}
-\sym(\mathcal{L}_{X_{0}}) & 0\\
0 & \sym(\mathcal{L}_{X_{0}})
\end{array}\right)M_{0}\alpha\\
 & =R_{\varepsilon}^{\ast}\frac{1}{2}\overline{\left[\left(\begin{array}{cc}
\mathcal{L}_{X_{0}} & 0\\
0 & -\mathcal{L}_{X_{0}}^{\ast}
\end{array}\right),\alpha M_{0}\right]}+\left(\begin{array}{cc}
\mathcal{L}_{X_{0}}^{\ast} & 0\\
0 & -\mathcal{L}_{X_{0}}
\end{array}\right)R_{\varepsilon}^{\ast}M_{0}\alpha+\\
 & \quad+R_{\varepsilon}^{\ast}\left(\begin{array}{cc}
-\sym(\mathcal{L}_{X_{0}}) & 0\\
0 & \sym(\mathcal{L}_{X_{0}})
\end{array}\right)M_{0}\alpha,
\end{align*}
where we have used Corollary \ref{cor:adjoint_m_accretive} in the
last equality. Now, on $\dom(\mathcal{L}_{X_{0}})\times\dom(\mathcal{L}_{X_{0}})$
we obtain
\begin{align*}
DL_{\varepsilon}+(DR_{\varepsilon})^{\ast}\to & \frac{1}{2}\left[\alpha M_{0},\left(\begin{array}{cc}
\mathcal{L}_{X_{0}}^{\ast} & 0\\
0 & -\mathcal{L}_{X_{0}}
\end{array}\right)\right]+\alpha M_{0}\left(\begin{array}{cc}
\mathcal{L}_{X_{0}} & 0\\
0 & -\mathcal{L}_{X_{0}}^{\ast}
\end{array}\right)+\\
 & +\alpha M_{0}\left(\begin{array}{cc}
-\sym(\mathcal{L}_{X_{0}}) & 0\\
0 & \sym(\mathcal{L}_{X_{0}})
\end{array}\right)+\frac{1}{2}\left[\left(\begin{array}{cc}
\mathcal{L}_{X_{0}} & 0\\
0 & -\mathcal{L}_{X_{0}}^{\ast}
\end{array}\right),\alpha M_{0}\right]+\\
 & +\left(\begin{array}{cc}
\mathcal{L}_{X_{0}}^{\ast} & 0\\
0 & -\mathcal{L}_{X_{0}}
\end{array}\right)M_{0}\alpha+\left(\begin{array}{cc}
-\sym(\mathcal{L}_{X_{0}}) & 0\\
0 & \sym(\mathcal{L}_{X_{0}})
\end{array}\right)M_{0}\alpha\\
= & \frac{1}{2}\left(-\left(\begin{array}{cc}
\mathcal{L}_{X_{0}}^{\ast} & 0\\
0 & -\mathcal{L}_{X_{0}}
\end{array}\right)M_{0}\alpha+\alpha M_{0}\left(\begin{array}{cc}
\mathcal{L}_{X_{0}} & 0\\
0 & -\mathcal{L}_{X_{0}}^{\ast}
\end{array}\right)-\right.\\
 & \qquad\left.-\alpha M_{0}\left(\begin{array}{cc}
\mathcal{L}_{X_{0}} & 0\\
0 & -\mathcal{L}_{X_{0}}^{\ast}
\end{array}\right)+\left(\begin{array}{cc}
\mathcal{L}_{X_{0}}^{\ast} & 0\\
0 & -\mathcal{L}_{X_{0}}
\end{array}\right)M_{0}\alpha\right)\\
= & 0
\end{align*}
and thus, it suffices to show that $DL_{\varepsilon}+(DR_{\varepsilon})^{\ast}$
is uniformly bounded in $\varepsilon.$ Using that $L_{\varepsilon}^{\ast}$
and $R_{\varepsilon}^{\ast}$ are uniformly bounded, the only term
we have to consider in the sum is 
\begin{align*}
\alpha M_{0}\left(\begin{array}{cc}
\mathcal{L}_{X_{0}} & 0\\
0 & -\mathcal{L}_{X_{0}}^{\ast}
\end{array}\right)L_{\varepsilon}^{\ast}+\left(\begin{array}{cc}
\mathcal{L}_{X_{0}}^{\ast} & 0\\
0 & -\mathcal{L}_{X_{0}}
\end{array}\right)R_{\varepsilon}^{\ast}M_{0}\alpha & =\alpha M_{0}\frac{1}{\epsilon}\left(-1+L_{\varepsilon}^{\ast}\right)+\frac{1}{\epsilon}(1-R_{\varepsilon}^{\ast})M_{0}\alpha\\
 & =\frac{1}{\varepsilon}\left(\alpha M_{0}L_{\varepsilon}^{\ast}-R_{\varepsilon}^{\ast}M_{0}\alpha\right)\\
 & =-\frac{1}{\varepsilon}\left([R_{\varepsilon}^{\ast},\alpha M_{0},L_{\varepsilon}^{\ast}]\right).
\end{align*}
This term, however, is uniformly bounded by Lemma \ref{lem:transmutator-bdd}.
\end{proof}
\begin{thm}
\label{thm:main}Problem (\ref{eq:prob}) is well-posed in the following
sense: There exists $\rho_{0}>0$ such that for each $\rho\geq\rho_{0}$
the operator 
\begin{equation}
\overline{\partial_{0}M_{0}+\tilde{M}_{1}+\overline{\left(D+\left(\begin{array}{cc}
0 & -\interior{d}^{\ast}\\
\interior{d} & 0
\end{array}\right)\right)}^{H}}^{L_{2,\rho}(\R;H)}\label{eq:correct op}
\end{equation}
is continuously invertible. Here 
\[
\tilde{M}_{1}\coloneqq M_{1}+\alpha\left(\overline{\left(\begin{array}{cc}
\nabla_{X_{0}} & 0\\
0 & \nabla_{X_{0}}
\end{array}\right)-\left(\begin{array}{cc}
\mathcal{L}_{X_{0}} & 0\\
0 & -\mathcal{L}_{X_{0}}^{*}
\end{array}\right)}\right)M_{0}+\alpha\overline{\left[\left(\begin{array}{cc}
\mathcal{L}_{X_{0}} & 0\\
0 & -\mathcal{L}_{X_{0}}^{*}
\end{array}\right),M_{0}\right]}+C
\]
and $C$ and $D$ are given as in Proposition \ref{prop:decomposition}.
Moreover, denoting by $S_{\rho}$ the inverse of (\ref{eq:correct op}),
we have that $S_{\rho}$ is causal and independent of the choice of
$\rho\geq\rho_{0}$ in the sense that for each $F\in H_{\rho,0}(\R;H)\cap H_{\mu,0}(\R;H)$
with $\rho,\mu\geq\rho_{0}$ we have that 
\[
S_{\rho}F=S_{\mu}F.
\]
\end{thm}

\begin{proof}
The claim follows from Corollary \ref{cor:perturbation}, where we
use that $\tilde{M}_{1}$ is bounded by Lemma \ref{lem:first-reduction}
and Proposition \ref{prop:decomposition} and that $\overline{\left(D+\left(\begin{array}{cc}
0 & -\interior{d}^{\ast}\\
\interior{d} & 0
\end{array}\right)\right)}^{H}$ is skew-selfadjoint by Proposition \ref{prop:sum-skew-selfadjoint}.
\end{proof}

\subsection{Cylindrical domains}

We shall consider cylindrical domains as two reference cases:
\begin{enumerate}
\item The infinite straight tube: $\Omega=\Sigma\times\mathbb{R}$, $\Sigma\subseteq\mathbb{R}^{2}.$
\item The finite straight tube: $\Omega=\Sigma\times]-1/2,1/2[$, $\Sigma\subseteq\mathbb{R}^{2},$
(with top and bottom identified this is a torus).
\end{enumerate}
We assume a metric tensor $g$ for $\Omega$. We discuss the Hypotheses
\ref{hyp:standard} in both cases for the vector field $X_{0}\coloneqq e_{3}.$ 

In both cases we recall that $\mathcal{L}_{e_{3}}$ is essentially
$\partial_{3}$ (up to lower order terms). In particular, $\mathcal{L}_{e_{3}}$
is quasi-skew-selfadjoint, since in both cases $\partial_{3}$ is
skew-selfadjoint (note that we impose periodic boundary conditions
in the second case). Moreover, the assumption 
\[
\sup\left\{ \|\nabla_{Y}e_{3}\|_{\infty}\,;\,Y\in T_{1,c}^{0}(\Omega),\:\|Y\|_{\infty}\leq1\right\} <\infty
\]
 is a constraint on the Riemannian metric $g$. Indeed, we compute
\[
\nabla_{Y}e_{3}=\nabla_{Y}\left(\sum_{i=1}^{n}g^{i}(e_{3})g_{i}\right)=\sum_{i=1}^{n}(\nabla_{Y}g^{i}(e_{3}))g_{i}+g^{i}(e_{3})\nabla_{Y}g_{i}
\]
for each $Y\in T_{1}^{0}(\Omega)$ and we require that the metric
$g$ is given such that 
\[
\sup\left\{ \|\nabla_{Y}e_{3}\|_{\infty}\,;\,Y\in T_{1,c}^{0}(\Omega),\:\|Y\|_{\infty}\leq1\right\} <\infty
\]
 is satisfied. Thus, we are left to discuss the assumption 
\[
(1\pm\epsilon\mathcal{L}_{X})[\Lambda^{k}(\Omega)\cap\dom(\mathcal{L}_{X})\cap\dom(d)]\cap\dom(\interior{d})\text{ is dense in }\dom(\interior d).
\]
We will do this in both cases separately.

\begin{itemize}

\item{\textbf{Case 1:}} 
\begin{prop}
\label{prop:staright_tube}For $|\eta|$ small enough we have
\[
\left(1+\eta\mathcal{L}_{e_{3}}\right)\left[\Lambda^{k}(\Omega)\cap\dom(\mathcal{L}_{e_{3}})\right]\cap\dom(\interior{d})\text{ is dense in }\dom(\interior d).
\]
\end{prop}

\begin{proof}
The set $C_{c}^{\infty}(\Omega)$ is dense in $\dom(\interior d)$
and -- as can be shown by a standard cut-off technique -- 
\begin{align*}
\interior Z\left(\Omega\right) & \coloneqq\left\{ \phi\in C^{\infty}\left(\Omega\right)\cap\dom\left(d\right)\,;\,\supp\left(\phi\right)\subseteq\tilde{\Sigma}\times\mathbb{R}\text{ for some }\tilde{\Sigma}\text{ relatively compact in }\Sigma\right\} \\
 & \subseteq\dom(\interior d).
\end{align*}
We show that $(1+\eta\mathcal{L}_{e_{3}})^{-1}f\in\interior{Z}(\Omega)$
for $f\in C_{c}^{\infty}(\Omega).$ Indeed, setting $u\coloneqq(1+\eta\mathcal{L}_{e_{3}})^{-1}f$
and $b\coloneqq\overline{\mathcal{L}_{e_{3}}-\partial_{3}}$, we infer
that 
\[
\left(1+\eta\partial_{3}\right)u+\eta bu=f.
\]
Since $b$ is a smooth multiplication operator, $u\in C^{\infty}(\Omega)\cap\dom(d)$
and $\supp\left(u\right)\subseteq\tilde{\Sigma}\times\mathbb{R}$
where $\tilde{\Sigma}\subseteq\Sigma$ is relatively compact such
that $\supp f\subseteq\tilde{\Sigma}\times\R$. Thus, we have in particular
\[
\interior z\left(\Omega\right)\coloneqq\left(1+\eta\mathcal{L}_{e_{3}}\right)^{-1}\left[C_{c}^{\infty}\left(\Omega\right)\right]\subseteq\interior Z\left(\Omega\right)\subseteq\dom\left(\interior d\right).
\]
and so
\[
\left(1+\eta\mathcal{L}_{e_{3}}\right)\left[\interior z\left(\Omega\right)\right]=C_{c}^{\infty}\left(\Omega\right)\text{ dense in }\dom\left(\interior d\right).
\]
Since
\[
\interior z\left(\Omega\right)\subseteq\Lambda^{k}(\Omega)\cap\dom(\mathcal{L}_{e_{3}}),
\]
the desired density property follows.
\end{proof}
\item{\textbf{Case 2:}}
\begin{prop}
We have
\[
\left(1+\eta\mathcal{L}_{e_{3}}\right)\left[\Lambda^{k}(\Omega)\cap\dom(\mathcal{L}_{e_{3}})\right]\cap\dom(\interior{d})\text{ is dense in }\dom\left(\interior d\right).
\]
\end{prop}

\begin{proof}
As in Proposition \ref{prop:staright_tube} we see that 
\begin{align*}
 & \interior Z\left(\Omega\right)\coloneqq\\
 & =\left\{ \phi\in C^{\infty}\left(\overline{\Omega}\right)\cap\dom\left(d\right)\,;\,\phi\left(\cdot,\cdot,-1/2\right)=\phi\left(\cdot,\cdot,1/2\right),\right.\\
 & \qquad\left.\supp\left(\phi\right)\subseteq\tilde{\Sigma}\times[-1/2,1/2]\text{ for some }\tilde{\Sigma}\text{ relatively compact in }\Sigma\right\} \\
 & \subseteq\dom\left(\interior d\right).
\end{align*}
We show that $(1+\eta\mathcal{L}_{e_{3}})^{-1}f\in\interior{Z}(\Omega)$
for $f\in C_{c}^{\infty}(\Omega).$ Indeed, setting $u\coloneqq(1+\eta\mathcal{L}_{e_{3}})^{-1}f$
and $b\coloneqq\overline{\mathcal{L}_{e_{3}}-\partial_{3}}$, we infer
that 
\[
\left(1+\eta\partial_{3}\right)u+\eta bu=f.
\]
Since $b$ is a periodic multiplication operator, we obtain $u\in\interior{Z}(\Omega)$
and thus, we have in particular
\[
\interior z\left(\Omega\right)\coloneqq\left(1+\eta\mathcal{L}_{e_{3}}\right)^{-1}\left[C_{c}^{\infty}(\Omega)\right]\subseteq\interior Z\left(\Omega\right)\subseteq\dom\left(\interior d\right).
\]
and so
\[
\left(1+\eta\mathcal{L}_{e_{3}}\right)\left[\interior z\left(\Omega\right)\right]=C_{c}^{\infty}(\Omega)\text{ dense in }\dom\left(\interior d\right).
\]
Since
\[
\interior z\left(\Omega\right)\subseteq\Lambda^{k}(\Omega)\cap\dom(\mathcal{L}_{e_{3}}),
\]
the desired density property follows.
\end{proof}
\end{itemize}
\begin{rem}
We have provided two examples, where the solution theory developed
in Subsection \ref{subsec:The-Equations-on} can be applied. Using
isometries between Riemannian manifolds, we can apply our solution
theory to a broader class of examples. More precisely, let $M$ and
$N$ be two Riemannian manifolds and $\Phi:M\to N$ be smooth and
orientation preserving. We denote the associated pull-back of cotangential
vectorfields (and tensors thereof) by $\Phi^{\ast}$ and the push-foward
of tangential vectorfields (and tensors thereof) by $\Phi_{\ast}$.
We assume that $\Phi$ is an isometry; that is,
\[
\Phi^{\ast}g_{N}=g_{M},
\]
where $g_{N}$ and $g_{M}$ denote the Riemannian metrics on $N$
and $M$, respectively. An easy computation shows that $\Phi^{\ast}$
commutes with the Hodge-star operator in the sense that 
\[
\Phi^{\ast}\ast_{N}=\ast_{M}\Phi^{\ast}
\]
and also with the exterior derivative. Thus, in particular 
\[
\Phi^{\ast}\left(\begin{array}{cc}
0 & -\interior{d}_{N}^{\ast}\\
\interior{d}_{N} & 0
\end{array}\right)=\left(\begin{array}{cc}
0 & -\interior{d}_{M}^{\ast}\\
\interior{d}_{M} & 0
\end{array}\right)\Phi^{\ast}.
\]
Moreover, $\Phi^{\ast}$ interacts with the Lie-derivative in the
following way 
\[
\Phi^{\ast}\mathcal{L}_{X}=\mathcal{L}_{\Phi_{\ast}^{-1}X}\Phi^{\ast}
\]
for a vectorfield $X$ on $N$. Hence, we can transform the equation
on $N$ via $\Phi^{\ast}$ into a corresponding equation on $M$.
Moreover, due to the isometry of $\Phi$, the condition on $M_{0}$,
defined on forms on $N$ (selfadjointness and positive definiteness)
carries over to the transformed operator $\Phi^{\ast}M_{0}\left(\Phi^{\ast}\right)^{-1}.$
Thus, if we can apply the solution theory to the problem posed on
the reference manifold $M$, we also derive the well-posedness on
the manifold $N$. Hence, by transforming the two reference situations
above, we can also treat the case of a infinite deformed pipe or a
deformed torus. 
\end{rem}

\subsection{An abstract localisation technique}

In this section we inspect a localisation technique, which will allow
us to glue together different open subsets of a Riemannian manifold.
Before we come to the concrete application, we will present this technique
in an abstract functional analytic setting. Throughout this section,
we consider the following setting:\\
Let $H$ be a Hilbert space and $U\subseteq H$ a closed subspace.
We denote the canonical embedding of $U$ into $H$ by $\iota_{U}:U\to H$
and remark that the adjoint $\iota_{U}^{\ast}:H\to U$ assigns each
element in $H$ its best approximation in $U$. Note that then $\iota_{U}\iota_{U}^{\ast}:H\to H$
is the orthogonal projector onto $U$ and $\iota_{U}^{\ast}\iota_{U}:U\to U$
is the identity on $U$. Moreover, let $A:\dom(A)\subseteq H\to H$
and $B:\dom(B)\subseteq U\to U$ be two densely defined closed linear
operators and we assume that $A^{\ast}$ leaves $U$ invariant. Finally,
let $S:H\to H$ be a bounded linear operator such that $\ran(S),\ran(S^{\ast})\subseteq U$,
$S[\dom(A)]\subseteq\dom(A),$ and $[S,A]$ is continuous.
\begin{example}
A typical example for the situation above is as follows: Let $\Omega\subseteq M$
an open subset of a Riemannian manifold and $\tilde{\Omega}\subseteq\Omega$
open. We set $H\coloneqq\Lambda_{2}^{k}(\Omega)\oplus\Lambda_{2}^{k+1}(\Omega)$
for some $k\in\mathbb{N}$ and $U\coloneqq\Lambda^{k}(\tilde{\Omega})\oplus\Lambda^{k+1}(\tilde{\Omega}).$
Moreover, let 
\[
A\coloneqq\left(\begin{array}{cc}
0 & -\interior{d}_{\Omega}^{\ast}\\
\interior{d}_{\Omega} & 0
\end{array}\right),\quad B\coloneqq\left(\begin{array}{cc}
0 & -\interior{d}_{\tilde{\Omega}}^{\ast}\\
\interior{d}_{\tilde{\Omega}} & 0
\end{array}\right),
\]
where $\interior{d}_{\Omega}$ and $\interior{d}_{\tilde{\Omega}}$
denote the exterior derivative with homogeneous Dirichlet boundary
conditions on $\Omega$ and $\tilde{\Omega}$, respectively. Moreover,
let $\phi:\Omega\to\R$ be smooth with $\supp\phi\subseteq\tilde{\Omega}$
and denote by $S$ the multiplication operator with $\phi$. 
\end{example}

We start with two useful observations.
\begin{lem}
\label{lem:commutator_S_and_A} $S^{\ast}[\dom(A^{\ast})]\subseteq\dom(A^{\ast})$
and $\overline{[A^{\ast},S^{\ast}]}=\overline{[S,A]}^{\ast}$.
\end{lem}

\begin{proof}
We recall from \prettyref{lem:transmutation} (a) 
\[
SA\subseteq AS+\overline{[S,A]}.
\]
Taking adjoints on both sides yields 
\[
A^{\ast}S^{\ast}=(SA)^{\ast}\supseteq(AS)^{\ast}+\overline{[S,A]}^{\ast}\supseteq S^{\ast}A^{\ast}+\overline{[S,A]}^{\ast}.
\]
The latter gives 
\[
S^{\ast}A^{\ast}\subseteq A^{\ast}S^{\ast}+\overline{[A,S]}^{\ast},
\]
which shows the claim.
\end{proof}
\begin{lem}
\label{lem:Extension}Assume that there exists a mapping $E:\dom(B)\to\dom(A)$
such that $\iota_{U}^{\ast}Ex=x$ and $\iota_{U}^{\ast}AEx=Bx$ for
each $x\in\dom(B).$ Let $v\in\dom(A^{\ast}).$ Then $\iota_{U}^{\ast}S^{\ast}v\in\dom(B^{\ast})$
with $B^{\ast}\iota_{U}^{\ast}S^{\ast}v=\iota_{U}^{\ast}A^{\ast}S^{\ast}v$.
\end{lem}

\begin{proof}
For $u\in\dom(B)$ we compute 
\begin{align*}
\langle Bu,\iota_{U}^{\ast}S^{\ast}v\rangle & =\langle S\iota_{U}Bu,v\rangle\\
 & =\langle S\iota_{U}\iota_{U}^{\ast}AEu,v\rangle\\
 & =\langle\iota_{U}\iota_{U}^{\ast}AEu,S^{\ast}v\rangle\\
 & =\langle AEu,S^{\ast}v\rangle\\
 & =\langle SAEu,v\rangle\\
 & =\langle ASEu,v\rangle+\langle\overline{[S,A]}Eu,v\rangle\\
 & =\langle Eu,\left(S^{\ast}A^{\ast}+\overline{[S,A]}^{\ast}\right)v\rangle\\
 & =\langle Eu,A^{\ast}S^{\ast}v\rangle.
\end{align*}
To finish the proof, observe that $A^{\ast}S^{\ast}v\in U$ and hence
\[
\langle Eu,A^{\ast}S^{\ast}v\rangle=\langle Eu,\iota_{U}\iota_{U}^{\ast}A^{\ast}S^{\ast}v\rangle=\langle\iota_{U}^{\ast}Eu,\iota_{U}^{\ast}A^{\ast}S^{\ast}v\rangle=\langle u,\iota_{U}^{\ast}A^{\ast}S^{\ast}v\rangle.
\]
This shows $\iota_{U}^{\ast}S^{\ast}v\in\dom(B^{\ast})$ and $B^{\ast}\iota_{U}^{\ast}S^{\ast}v=\iota_{U}^{\ast}A^{\ast}S^{\ast}v.$
\end{proof}
We now come to the concrete application. Let $\Omega\subseteq M$
be open for a smooth Riemannian manifold $M$ of odd dimension $n$.
Moreover, let $1\leq k<n$ and set $H\coloneqq\Lambda^{k}(\Omega)\times\Lambda^{k+1}(\Omega).$
We assume that $M_{0},M_{1}\in L(H)$ such that $M_{0}$ is selfadjoint
and $M_{0}\geq c>0$. Moreover, let $X_{0}\in T_{1}^{0}(M)$ and $\alpha\in L_{\infty}(\Omega;\R)$.
We again consider an equation of the form 
\[
\left(\partial_{0}M_{0}+\alpha\left(\begin{array}{cc}
\nabla_{X_{0}} & 0\\
0 & \nabla_{X_{0}}
\end{array}\right)M_{0}+M_{1}+\left(\begin{array}{cc}
0 & -\interior{d}^{*}\\
\interior{d} & 0
\end{array}\right)\right)U=F
\]
but now we assume that the vectorfield $X_{0}$ and the function $\alpha$
are supported in an open subset $\tilde{\Omega}\subseteq\Omega$.
Moreover, we assume that $M_{0}$ and $M_{1}$ are local operators
(e.g. multiplication operators) and that the hypotheses are satisfied
on two open subdomains $\tilde{\Omega}\subseteq\Omega$ and $\hat{\Omega}\subseteq\Omega$
with $\Omega=\hat{\Omega}\cup\tilde{\Omega}$ and $\alpha$ and $X_{0}$
vanish on $\hat{\Omega}.$ Then we can solve the problem separately
on $\tilde{\Omega}$ and $\hat{\Omega}.$ We now employ the localisation
technique to illustrate how this yields a solution theory for the
original problem on $\Omega$. The crucial point for showing the well-posedness
is the essential skew-selfadjointness of $D+\left(\begin{array}{cc}
0 & -\interior{d}^{\ast}\\
\interior{d} & 0
\end{array}\right)$ with $D\coloneqq\skew\alpha M_{0}\left(\begin{array}{cc}
\mathcal{L}_{X_{0}} & 0\\
0 & -\mathcal{L}_{X_{0}}^{*}
\end{array}\right)$. In order to show this, we employ the result above in the following
two situations: Let $\phi:\Omega\to[0,1]$ be smooth with $\supp\phi\subseteq\tilde{\Omega}$
and $\supp X_{0}\subseteq[\phi=1].$ Moreover, we set 
\[
A\coloneqq\overline{\skew\alpha M_{0}\left(\begin{array}{cc}
\mathcal{L}_{X_{0}} & 0\\
0 & -\mathcal{L}_{X_{0}}^{*}
\end{array}\right)+\left(\begin{array}{cc}
0 & -\interior{d}^{\ast}\\
\interior{d} & 0
\end{array}\right)}
\]

\begin{itemize}
\item $U_{1}\coloneqq\Lambda^{k}(\tilde{\Omega})\times\Lambda^{k+1}(\tilde{\Omega}),$
$B_{1}\coloneqq\overline{\iota_{U_{1}}^{\ast}\skew\alpha M_{0}\left(\begin{array}{cc}
\mathcal{L}_{X_{0}} & 0\\
0 & -\mathcal{L}_{X_{0}}^{*}
\end{array}\right)\iota_{U_{1}}+\left(\begin{array}{cc}
0 & -\interior{d}_{\tilde{\Omega}}^{\ast}\\
\interior{d}_{\tilde{\Omega}} & 0
\end{array}\right)}$ and $S_{1}\coloneqq\phi(m)$ the multiplication operator with $\phi.$
In order to apply \prettyref{lem:Extension} we have to assume the
existence of an extension operator $E_{1}:\dom(B_{1})\to\dom(A),$
which is an implicit regularity assumption for the part of the boundary
of $\tilde{\Omega}$ in the interior of $\Omega.$ 
\item $U_{2}\coloneqq\Lambda^{k}(\hat{\Omega})\times\Lambda^{k+1}(\hat{\Omega}),$
$B_{2}\coloneqq\left(\begin{array}{cc}
0 & -\interior{d}_{\hat{\Omega}}^{\ast}\\
\interior{d}_{\hat{\Omega}} & 0
\end{array}\right)$ and $S_{2}\coloneqq1-\phi(m)$ the multiplication operator with $1-\phi.$
Again, we have to assume the existence of an extension operator $E_{2}:\dom(B_{2})\to\dom(A),$
which is an implicit regularity assumption for the part of the boundary
of $\hat{\Omega}$ in the interior of $\Omega.$
\end{itemize}
\begin{thm}
Assume that $B_{1}$ and $B_{2}$ are skew-selfadjoint operators on
$U_{1}$ and $U_{2}$, respectively. Then $A$ is skew-selfadjoint
on $H$.
\end{thm}

\begin{proof}
Since $A$ is clearly skew-symmetric, it suffices to prove $\dom(A^{\ast})\subseteq\dom(A).$
For doing so, let $v\in\dom(A^{\ast}).$ By \prettyref{lem:Extension}
we have that $\iota_{U_{1}}^{\ast}S_{1}v\in\dom(B_{1}^{\ast})=\dom(B_{1})$
as well as $\iota_{U_{2}}^{\ast}S_{2}v\in\dom(B_{2}^{\ast})=\dom(B_{2}).$
Next, we observe that $S_{1}v,S_{2}v\in\dom(A).$ Indeed, since $\iota_{U_{1}}^{\ast}S_{1}v\in\dom(B_{1})$
we find a sequence $\left(\psi_{n}\right)_{n}$ in $\dom\left(\iota_{U_{1}}^{\ast}\skew\alpha M_{0}\left(\begin{array}{cc}
\mathcal{L}_{X_{0}} & 0\\
0 & -\mathcal{L}_{X_{0}}^{*}
\end{array}\right)\iota_{U_{1}}\right)\cap\dom\left(\begin{array}{cc}
0 & -\interior{d}_{\tilde{\Omega}}^{\ast}\\
\interior{d}_{\tilde{\Omega}} & 0
\end{array}\right)$ with 
\[
\psi_{n}\to\iota_{U_{1}}^{\ast}S_{1}v\text{ and }\left(\iota_{U_{1}}^{\ast}\skew\alpha M_{0}\left(\begin{array}{cc}
\mathcal{L}_{X_{0}} & 0\\
0 & -\mathcal{L}_{X_{0}}^{*}
\end{array}\right)\iota_{U_{1}}+\left(\begin{array}{cc}
0 & -\interior{d}_{\tilde{\Omega}}^{\ast}\\
\interior{d}_{\tilde{\Omega}} & 0
\end{array}\right)\right)\psi_{n}\to B_{1}\iota_{U_{1}}^{\ast}S_{1}v\quad(n\to\infty)
\]
in $U_{1}.$ Take a smooth function $\zeta:\Omega\to\R$ with $\supp\zeta\subseteq\tilde{\Omega}$
and $\zeta=1$ on $\supp\phi.$ Moreover, we denote by $\tilde{\psi}_{n}$
the extension of $\psi_{n}$ to $\Omega$ by $0$. Then 
\[
\zeta\tilde{\psi}_{n}\to S_{1}v
\]
in $H$ and 
\begin{align*}
 & \left(\skew\alpha M_{0}\left(\begin{array}{cc}
\mathcal{L}_{X_{0}} & 0\\
0 & -\mathcal{L}_{X_{0}}^{*}
\end{array}\right)+\left(\begin{array}{cc}
0 & -\interior{d}^{\ast}\\
\interior{d} & 0
\end{array}\right)\right)\zeta\tilde{\psi}_{n}\\
 & =\iota_{U_{1}}\left(\iota_{U_{1}}^{\ast}\skew\alpha M_{0}\left(\begin{array}{cc}
\mathcal{L}_{X_{0}} & 0\\
0 & -\mathcal{L}_{X_{0}}^{*}
\end{array}\right)\iota_{U_{1}}+\left(\begin{array}{cc}
0 & -\interior{d}_{\tilde{\Omega}}^{\ast}\\
\interior{d}_{\tilde{\Omega}} & 0
\end{array}\right)\right)\zeta\psi_{n}\\
 & =\iota_{U_{1}}\zeta\left(\iota_{U_{1}}^{\ast}\skew\alpha M_{0}\left(\begin{array}{cc}
\mathcal{L}_{X_{0}} & 0\\
0 & -\mathcal{L}_{X_{0}}^{*}
\end{array}\right)\iota_{U_{1}}+\left(\begin{array}{cc}
0 & -\interior{d}_{\tilde{\Omega}}^{\ast}\\
\interior{d}_{\tilde{\Omega}} & 0
\end{array}\right)\right)\psi_{n}+\\
 & \quad+\iota_{U_{1}}\left[\iota_{U_{1}}^{\ast}\skew\alpha M_{0}\left(\begin{array}{cc}
\mathcal{L}_{X_{0}} & 0\\
0 & -\mathcal{L}_{X_{0}}^{*}
\end{array}\right)\iota_{U_{1}}+\left(\begin{array}{cc}
0 & -\interior{d}_{\tilde{\Omega}}^{\ast}\\
\interior{d}_{\tilde{\Omega}} & 0
\end{array}\right),\zeta\right]\psi_{n}.
\end{align*}
Since $\zeta$ is smooth, the commutator $\left[\iota_{U_{1}}^{\ast}\skew\alpha M_{0}\left(\begin{array}{cc}
\mathcal{L}_{X_{0}} & 0\\
0 & -\mathcal{L}_{X_{0}}^{*}
\end{array}\right)\iota_{U_{1}}+\left(\begin{array}{cc}
0 & -\interior{d}_{\tilde{\Omega}}^{\ast}\\
\interior{d}_{\tilde{\Omega}} & 0
\end{array}\right),\zeta\right]$ is continuous and hence, the latter term converges in $H$ as $n\to\infty.$
Thus, $S_{1}v\in\dom(A)$ by definition of $A$. In the same way one
obtains $S_{2}v\in\dom(A)$ and thus, $v=S_{1}v+S_{2}v\in\dom(A),$
which completes the proof.
\end{proof}

\section*{Acknowledgement}

The first author is indebted to Martin Berggren and Linus Hägg for
the inspiration and the opportunity for interesting extended discussions.

\end{document}